\documentclass[12pt]{amsart}
\usepackage{amsmath,amsthm,amssymb,subeqnarray}
\usepackage{graphicx,epic}
\usepackage{mathrsfs,stackrel}
\usepackage{hyperref,color}
\usepackage{bookmark}
\usepackage{mathtools,abraces,yhmath}
\usepackage[dvipsnames]{xcolor}

\usepackage{enumitem}

\usepackage{float}
\usepackage[utf8]{inputenc} 
\newtheorem{theorem}{Theorem}
\newtheorem{remark}{Remark}
\newtheorem{proposition}{Proposition}[section]
\newtheorem{lemma}[proposition]{Lemma}
\newtheorem{corollary}[proposition]{Corollary}

\newtheorem{definition}[proposition]{Definition}
\newtheorem{claim}{Claim}
\newtheorem{notation}{Notation}

\renewcommand{\ae}{{\it a.e. }}
\input{macros.sty}
\newcommand{\eg}{\emph{e.g.}\ }
\def\fibar{\ol{\varphi}_{\be}}
\def\1{\BBone}

\def\ogpsi{\ovec{\gramat\psi}}
\def\gt{{\gramat t}}
\def\gz{{\gramat z}}
\def\gxi{{\gramat \xi}}

\def\dom{{d_{\Om}}}
\def\Lipom{\CC^{+1}(\Om)}
\def\ga{\gamma}

\def\fibe{{\varphi_{\be}}}

\begin{document}

\title[Generalized Curie-Weiss Potts models and quadratic pressure]{Generalized
Curie-Weiss Potts models and quadratic pressure in ergodic theory}

\author{Renaud Leplaideur}
\address{ISEA,
Universit\'e de Nouvelle Calédonie,
145, Avenue James Cook - BP R4
         98 851 - Nouméa Cedex. Nouvelle Cal\'edonie}
    
\address{LMBA, UMR6205
Universit\'e de Brest.}
\email{renaud.leplaideur@unc.nc,
\url{http://rleplaideur.perso.math.cnrs.fr/} }

\author{Frédérique Watbled} 
\address{LMBA, UMR 6205, Universit\'e de Bretagne Sud,  Campus de Tohannic, BP 573, 56017 Vannes, France.}
\email{frederique.watbled@univ-ubs.fr}

\date{Version of \today}
\thanks{F. Watbled
 thanks the IRMAR, CNRS UMR 6625, University of Rennes 1, for its hospitality.}
\thanks{The authors thank the Centre Henri Lebesgue ANR-11-LABX-0020-01 for creating an attractive mathematical environment}

\begin{abstract}
We extend results on quadratic pressure and convergence of Gibbs mesures from \cite{Leplaideur-Watbled1} to the Curie-Weiss-Potts model. We define the notion of equilibrium state for the quadratic pressure and show that under some conditions on the maxima for some auxiliary function, the Gibbs measure converges to a convex combination of eigen-measures for the Transfer Operator. 
This extension works for  dynamical systems defined by {infinite-to-one} maps. As an example, we compute the equilibrium for {the mean-field} $XY$ model as the number of particles goes to $+\8$.
\end{abstract}

\maketitle

\section{Introduction}\label{sec:intro}
\subsection{Background, main motivations, open questions}
In a recent work (\cite{Leplaideur-Watbled1}) {the} authors defined the notion of quadratic pressure associated to some potential $\psi$ for the symbolic dynamics  $\{0,1\}^{\N}$ (with the shift map). The motivation for that was to study similarities and differences between phase transitions in Ergodic Theory on the one hand and in Probability and Statistical Mechanics on the other hand. 

The authors pointed out that the Curie-Weiss model in Probability theory can be linked to Ergodic Theory with the quadratic equilibriums. More precisely, it was shown that \emph{Probability Gibbs Measures} (PGM for short) converge as the number of sites goes to $+\8$ to a convex combination of \emph{Dynamical Conformal Measures} associated to the invariant measures which maximize the quadratic pressure. 

In the present paper,  Theorem \ref{thmpressquadra} and \ref{th-GCWP} extend results of \cite{Leplaideur-Watbled1}  to the Curie-Weiss-Potts model. This is a natural question as the Curie-Weiss-Potts model is a kind of generalization of Curie-Weiss model. In view to give an application to the $XY$-model (see Section \ref{sec-vlasov}), the statement is done for dynamical systems which are not necessarily {finite-to-one}. 

For that goal, the notion of entropy needs to be made precise. This notion has already been investigated  and we mention \eg a series of works \cite{CioLop1,CioLop2, lopesentrop1, lopesentrop2}  and more recently \cite{GKLM}. 
We re-employ the idea to define the entropy as the Fenchel-Legendre transform of the pressure function and  to link it to a min-max problem. The main assumption in \cite{GKLM} is the existence of the spectral gap for the transfer operator and the purpose of Theorem \ref{th-rayonana} is thus to state this spectral gap. 

\bigskip

Some version of the transfer operator  with infinite-to-one map has already been studied in \cite{lopesentrop2}. We point out that in our case we have a more flexible operator as the transition depends on the two first coordinates (see below). Actually, we believe that it can easily be extended to the case where transitions only depend on finitely many coordinates, which is what should be a natural extension of the notion of subshift of finite type with infinite (and uncountable) alphabet. 
In { other words}, the transfer operator in \cite{lopesentrop2} is the one for a full-shift of finite type whereas our is the one for more general irreducible subshift of finite type. 
Furthermore, we study here the regularity of the spectral radius\footnote{This is one key point in our work.}, in particular for the multi-dimensional case,   and we did not find any reference of that problem in \cite{lopesentrop2}. 

\bigskip

Theorem \ref{thmpressquadra} is where we make the link between equilibrium states for the linear pressure and the equilibrium states for the quadratic pressure. This result goes in the same direction as \cite{Leplaideur-Watbled1} and more recently \cite{Buzzi-Leplaideur}. Equilibrium states for quadratic pressure are equilibrium states for the linear pressure but with a change of the parameter. 

\bigskip

Theorem \ref{th-GCWP}  is where we make links between Dynamical Gibbs Measures and Probabilistic Gibbs Measures. It deals with convergence of the PGM to a convex combination of eigen-measures for the Transfer Operator. One of the key points is the Laplace method. We point out here a big difference between the 1d case and the multi-dimensional case. In the 1d-case, the Laplace method can be applied even if the Hessian at maximal points is degenerated. \\
{ We remind that Laplace method deals with integrals of the form $\int_{a}^{b}f(t)e^{n\phi(t)}dt$ and gives an equivalent of this quantity as $n$ goes to $+\8$. This equivalent only involves values $c_{i}$ where $\phi$  is maximal. 
For each $c_{i}$ one gets an expression in $n$ depending on how flat $\phi$ is closed to $c_{i}$ and also on how $f$ behaves closed to $c_{i}$. \\
The crucial point is that in dimension 1, we get an expression  at each $c_{i}$ and we can compare them. We remind that roughly speaking, it was shown in \cite{Leplaideur-Watbled1} that only  the maxima where $\phi$ is the flattest yield a positive contribution for the limit of the PGM.\\
On the contrary, this comparison between maxima does not seem to be (easily) possible if we deal with an integral in higher dimension, unless all the maxima have a non-degenerated Hessian. The consequence in our problem is that we can precisely determine what is the convex combination for the limit of the PGM, only if all the maxima are non-degenerated. }

This naturally leads to ask for which $\ogpsi$ the Hessian is non-degenerated. We have no idea for the answer yet. As we will see below, the mean-field $XY$ model has degenerate Hessian 
 {but it reduces to a one dimensional problem and we can deal with it}. 

\subsection{Settings and results}
\subsubsection{Shift with (possibly) infinite alphabet}
Let $(E,d)$ be a compact metric space, let $\rho$ be a Borel probability measure on $E$ with full support. We assume that $\rho$ satisfies the following assumption
$$\textbf{(H):}
\textrm{ for any sufficiently small }\eps>0, x\mapsto \rho(B(x,\eps))\textrm{ is continuous.}$$  
We consider a map $A:E\times E\to [0,1]$ called the transition function, which satisfies the following properties:
\begin{description}
\item[(A1)] $A$ is continuous with values in $\{0,1\}$. 
\item[(A2)] $A$ is Lipschitz continuous with respect to the second variable with Lipschitz constant $\Lip(A)$. 
\item[(A3)] $A$ generates some mixing in the following sense. 
\begin{multline}
\label{equ-mixingA}
\exists N\in\bb N,\, \forall\,n\ge N,\ \forall\, a,b\in E,\, \exists z_{1},\ldots ,z_{n-1}\in E\\
\textrm{ such that } A(a,z_{1})A(z_{1},z_{2})\ldots A(z_{n-1},b)=1.
\end{multline}
\end{description}

\begin{remark}
\label{rem-Atransimatrix}
The assumption (A1) yields that $A$ is constant with value 0 or 1 on each connected component of $E\times E$.  
The assumption (A3) implies in particular that $A$ is not identically null.
\end{remark}

We define $\Om\subset E^{\N}$ in the following way:
$$\Om:=\left\{x=x_{0}x_{1}x_{2}\ldots\in E^{\N}; \ \forall\,i\in\bb N,\ A(x_{i},x_{i+1})>0\right\}.$$
The shift map $\sigma:\Om\to\Om$ is defined by
$$\sigma(x_{0}x_{1}x_{2}\ldots)=x_{1}x_{2}\ldots.$$
Note that if $E$ is connected, \eg $E=[0,1]$, then $A\equiv 1$. If $E$ is a finite set $\{1,\ldots k\}$, then $\Om$ is the subshift of finite type with transition matrix having entries $A(i,j)$. 

For $n\geq 1$, let $\Om_{n}$ be the set of words $z_{1}\ldots z_{n}$ with $\disp\prod_{i=1}^{n-1} A(z_{i},z_{i+1})=1$.
For $a$ and $b$ in $E$, let $\Om_{n-1}(a,b)$ be the set of words $z_{1}\ldots z_{n-1}$ 
in $\Om_{n-1}$ with $A(a,z_1)=A(z_{n-1},b)=1$. Assumption (A3) on $A$ means that for every $a$, $b$ in $E$, for every $n\ge N$, $\Om_{n-1}(a,b)\neq\emptyset$. It implies in particular
that for every $a$ in $E$, there always exist $u$, $v$ in $E$ such that $A(a,u)=1$ and $A(v,a)=1$.
We denote by $\Om_{n}(b)$ the set of words $z_{0}\ldots z_{n-1}$ 
in $\Om_{n}$ with $A(z_{n-1},b)=1$. 

We set $\P=\rho^{\otimes\N}$. The distance over $\Om$ is defined by 
$$d_{\Om}(x,y)=\sum_{n=0}^{+\8}\frac{d(x_{n},y_{n})}{2^{n+1}}.$$
We notice that for any $a$ in $\Omega_n$,
$$d_{\Om}(ax,ay)=\frac{1}{2^n}d_{\Om}(x,y)$$
and that
$$d_{\Om}(\sigma^n x,\sigma^n y)=2^n\pare{d_{\Om}(x,y)-\sum_{k=0}^{n-1}\frac{d(x_{k},y_{k})}{2^{k+1}}}.$$

We denote by $\CC^{0}(\Om)$, respectively $\CC^{+1}(\Om)$, the set of continuous,
respectively Lip\-schitz continuous, functions from $\Om$ to $\bb R$,
equipped respectively with the norms
$$\|\phi\|_\infty=\underset{x\in\Omega}{\max}|\phi(x)|, \qquad
\|\phi\|_L=\|\phi\|_\infty+\Lip(\phi),$$
where $\Lip(\phi)$ stands for the Lipschitz constant of $\phi$. 
We recall that the spaces $(\CC^{0}(\Om),\|\cdot\|_\infty)$ and 
$(\CC^{+1}(\Om),\|\cdot\|_L)$ are Banach spaces.
We set $M(\Omega)$ the space of probability measures on $\Omega$ and recall that by the Riesz representation theorem, the map $\mu\mapsto (f\mapsto \int f\,d\mu)$ is a bijection 
between $M(\Om)$ and 
$$\{l\in \CC^{0}(\Om)^*; l(\BBone)=1 \textrm{ and } l(f)\geq 0 \textrm{ whenever }f\geq 0\}.$$  
A measure $\mu$ is $\sigma$-invariant if $\mu(\sigma^{-1}(B))=\mu(B)$ for all Borel sets $B$.
The set $M_\sigma(\Omega)$ is the space of $\sigma$-invariant probability measures on $\Om$. Both $M(\Omega)$ and $M_\sigma(\Omega)$ are convex and compact for the weak star topology.

The transfer operator associated to $\phi:\Om\to\R$ (Lipschitz continuous) is the linear operator defined by 
$$\CL_{\phi}(f)(\om)=\int_E e^{\phi(t\om)}A(t,\om_{0})f(t\om)\,d\rho(t).$$

Theorem \ref{th-rayonana} states several properties on the spectrum of the transfer operator. To properly state the theorem we need to introduce some more quantities.

The operator $\CL_{\phi}$ acts on $\CC^{0}(\Om)$ and on $\Lipom$.
The spectral radius of $\CL_{\phi}$ on $\CC^{0}(\Om)$, denoted by $r_{\phi}$, 
is a simple eigenvalue of the adjoint operator $\CL_{\phi}^\star$
acting on the space of Radon measures on $\Om$,
and the conformal measure $\nu_\phi$ is the unique probability eigen-measure associated to the eigenvalue $r_\phi$. It is also a simple eigenvalue of $\CL_{\phi}$ acting on $\Lipom$, with a
positive eigenfunction $G_\phi$ such that the measure $\mu_\phi=G_\phi\nu_\phi$ is a probability measure. We call $\mu_\phi$ the dynamical Gibbs measure (DGM for short) associated to $\phi$.

If $\gz$ belongs to $\R^{q}$ and $\psi_{i}$, $i=1,\ldots q$ are in $\Lipom$ one sets $\ogpsi:=(\psi_{1},\ldots,\psi_{q})$ and $\gz\cdot\ogpsi:=\sum_{i=1}^{q}z_{i}\psi_{i}$. We note $||\gz||$ the Euclidean norm of $\gz$
$$||\gz||^{2}=\sum_{i=1}^q z_{i}^{2}.$$

\begin{definition}
\label{def-entropyclasse}
For fixed $\ogpsi$ and $\gz\in\R^{q}$, one sets 
$$\CH(\gz,{ \ogpsi}):=\inf_{\gt\in\R^{q}}\left\{\log r_{\gt\cdot\ogpsi}-\gt\cdot\gz\right\}.$$
and 
$$I(\ogpsi):=\left\{\int \ogpsi\,d\mu,\ \mu\in M_\sigma(\Omega)\right\}.$$  
\end{definition}
Note that $I(\ogpsi)$ is a closed convex subset of $\R^{q}$. { Moreover,  $\gz\mapsto\CH(\gz,\ogpsi)$ is upper semi-continuous, as it is an infimum of affine functions.}

\begin{theorem}
\label{th-rayonana}
For any $\phi\in \CC^{+1}(\Om)$, $r_{\phi}$ is a simple single dominating eigenvalue. 
Moreover, 
for any $\ogpsi$ with $\psi_{i}\in \Lipom$, the map
$\CP:\gt\mapsto \log r_{\gt\cdot\ogpsi}$ is infinitely differentiable. 

Furthermore, 
\begin{enumerate}
\item $\CH(\gz,\ogpsi)=-\infty$ if $\gz\notin I(\ogpsi)$,
\item $\CH(\gz,\ogpsi)$ is finite if $\gz\in \nabla \CP(\bb R^q)$.
\end{enumerate}
\end{theorem}

We call \emph{pressure function} for $\ogpsi$ the map $\gt\mapsto \CP(\gt)$. 

\subsubsection{ Quadratic Pressure}
For fixed $\ogpsi\in \Lipom^{q}$ and for $\beta\geq 0$, $\gt$, $\gz$ in $\bb R^q$, we set
$$\varphi_{\be}(\gt):=-\frac\be2||\gt||^{2}+\log r_{\be \gt\cdot\ogpsi}\text{ and }\fibar(\gz):=\CH(\gz,\ogpsi)+\frac\be2||\gz||^{2}.$$
\begin{notation}
We set $\CH_{top}:=\log r_{0}$.
\end{notation}

\begin{definition}
\label{def-relativentropyquadrpress}
For $\mu$ in $M_\sigma(\Omega)$, the \emph{entropy} 
is the quantity 
$$\wh\CH(\mu):={\inf_{\ogpsi\in \Lipom^{q}}\CH\left(\int\ogpsi\,d\mu,\ogpsi\right).}$$

Let $\ogpsi\in\Lipom^{q}$ be fixed. 
The quantity $$\CP_{2}(\be)=\sup_{\mu}\left\{\wh\CH(\mu)+\frac\be2\left\|\int\ogpsi\,d\mu\right\|^{2}\right\}$$ is referred to as the quadratic pressure function for $\ogpsi$. 
\end{definition}

\begin{remark}
\label{rem-whusc}
In \cite{GKLM} the authors define the entropy by setting
$$h_{\cal X}(\mu):=\inf_{A\in \cal X(\Omega)}\pare{\log r_A-\int A\,d\mu}\textrm{ for }\mu\in M_\sigma(\Omega)$$
and the pressure by setting
$$\rm{Pr}(B):=\sup_{\mu\in M_\sigma(\Omega)}\pare{h_{\cal X}(\mu)+\int B\,d\mu}\textrm{ for }B\in \cal X(\Omega),$$
where $\cal X(\Omega)$ is a suitable space of potentials $\Omega\to\bb R$.
We notice that our definition of entropy is the same as theirs with $\cal X(\Omega)=\Lipom$, whereas 
our definition of pressure is linked to theirs by
$\cal P(\gt)=\rm{Pr}(\gt\cdot\ogpsi)$, where $\ogpsi$ is fixed in $\Lipom^{q}$.
We point out that $\mu\mapsto\wh\CH(\mu)$ is upper semi-continuous as the infimum over a family of upper semi-continuous functions. 
$\blacksquare$\end{remark}

From Theorem \ref{th-rayonana} we can use the work of Giulietti and \emph{ al} in \cite[th.F]{GKLM}. We emphasize that the key point in their work is the spectral decomposition of the transfer operator, which is stated in Theorem \ref{th-rayonana}. Then, we deduce the existence of the DGM which is the unique equilibrium state for $\phi$. 

Stated with our settings, we get that 
for any  $\ogpsi\in\Lipom^{q}$  and for any $\gramat t$ there is a unique invariant measure $\mu_{\gt.\ogpsi}$ which maximizes $\disp\wh\CH(\mu)+\int\gt\cdot\ogpsi\,d\mu$. Moreover, 
\begin{equation}\label{unique}
\CP(\gt)=\wh\CH(\mu_{\gt\cdot\ogpsi})+\int\gt\cdot\ogpsi\,d\mu_{\gt\cdot\ogpsi}.
\end{equation}

Convexity for the multi-dimensional pressure function $\gt\mapsto \CP(\gt)=\sup_{\mu}\{\wh\CH(\mu)+\int\gt\cdot\ogpsi\,d\mu\}$ and differentiability obtained from Theorem \ref{th-rayonana} yield that for every $\gt$ and every $i$, $\disp\frac{\partial \log r_{\gt\cdot\ogpsi}}{\partial t_{i}}=\int\psi_{i}\,d\mu_{\be\gt\cdot\ogpsi}.$

\bigskip

\begin{theorem}\label{thmpressquadra}
\emph{\textbf{Equilibrium states for the quadratic pressure.}}
For any  $\ogpsi\in\Lipom^{q}$ and for any $\be\ge 0$ the invariant probability measures which ma\-xi\-mi\-ze $\CP_{2}(\be)$ are the dynamical Gibbs measures $\mu_{\be\gramat t\cdot\ogpsi}$ where the $\gramat t$'s are the maxima for $\varphi_{\be}$.
\end{theorem}

This result goes in the same direction as the ones from \cite{Leplaideur-Watbled1} and \cite{Buzzi-Leplaideur}. However, we point out some interesting difference here: in the higher-dimensional case, there may be infinitely many measures which maximize the quadratic pressure. This is actually the case for $XY$-model (see below Remark \ref{rem-infinitequilquadra}).

\subsubsection{Generalized Curie-Weiss-Potts Hamiltonian}
For $\phi\in\Lipom$, we remind that $S_{n}(\phi)$ stands for $\phi+\ldots+\phi\circ \s^{n-1}$. With previous notations, $S_{n}(\ogpsi)$ is the vector with coordinates $S_{n}(\psi_{i})$. 
Then, the Generalized Curie-Weiss-Potts Hamiltonian is defined for  $\omega\in\Om$ by
$$H_{n}(\om):=-\frac1{2n}\|S_n(\ogpsi)(\omega)\|^2.$$

We define the probabilistic Gibbs measure (PGM for short) $\mu_{n,\beta}$ on $\Om$ by
\begin{equation}\label{eq-def-gibbsCWP}
\mu_{n,\beta}(d\omega):=\frac{e^{-\beta H_{n}(\om)}}{Z_{n,\beta}}\bb P(d\omega)
=\frac{e^{\frac{\be}{2n}\|S_{n}(\ogpsi)(\om)\|^2}}{Z_{n,\beta}}\bb P(d\omega),
\end{equation}
where $Z_{n,\beta}$ is the suitable normalization factor.

If $P_n$, $P$ are probability measures in $M(\Omega)$, we say that $P_n$ converges weakly to $P$ if $\int_\Omega f\,dP_n\rightarrow \int_\Omega f\, dP$ for each $f$ in $\CC^{0}(\Om)$.
As $\CC^{+1}(\Om)$ is dense in $\CC^{0}(\Om)$, this is equivalent to $\int_\Omega f\,dP_n\rightarrow \int_\Omega f\, dP$ for each $f$ in $\CC^{+1}(\Om)$. 

\begin{theorem}{{\bf Generalized Curie-Weiss-Potts model.}}
\label{th-GCWP}

One dimensional case: if $q=1$, then the PGM $\munbe$ converges weakly to a convex combination of the conformal measures $\nu_{\be t\psi}$ associated to the $\mu_{\be t\psi}$'s from Theorem \ref{thmpressquadra} as $n$ goes to $+\infty$. 

Higher dimensional case: for $q>1$, if $\varphi_{\be}$ attains its maximum only on non-degenerated points (\ie $d^{2}\varphi_{\be}$ is invertible), then they are finitely many and the PGM $\munbe$ converges weakly to a convex combination of the conformal measures $\nu_{\be\gt\cdot\ogpsi}$ associated to the $\mu_{\be\gt\cdot\ogpsi}$'s, where the $\gt$'s are the maxima for $\varphi_{\be}$. 
\end{theorem}

\begin{remark}
\label{rm-classicCWP}
The ``classical'' Curie-Weiss-Potts model (see Theorem 2.1 of \cite{EllisWang90}) is obtained by taking $E=\{1,\ldots q\}$, $\psi_{i}=\1_{[i]}$ and $A(i,j)=1$ for every pair $(i,j)$. 
$\blacksquare$\end{remark}

We point out that the mean-field $XY$ model (see Section \ref{sec-vlasov}) is an example where $\varphi_{\be}$ atteins its maximum on a infinite set of points, and for all of them the Hessian is degenerated. 

\subsection{Plan of the paper}
In Section \ref{sec-proofthrayo} we prove Theorem \ref{th-rayonana}. As we said above, the main ingredient is to define and study the spectrum of the Transfer Operator. We prove this operator has a spectral gap and that the spectral radius is a simple isolated dominating eigenvalue. This allows to define the notion of conformal measure. 

In Section \ref{sec-proofthcwp} we prove Theorem \ref{thmpressquadra}. The main ingredient is to define two auxiliary functions, $\varphi_{\be}$ and $\fibar$, to show that one is always bigger than the other one, but both have the same maxima (arising for the same points). Maximal value for $\fibar$ equals the quadratic pressure but maxima for $\varphi_{\be}$ are easier to detect. 
{ Part of the difficulty in this section comes from our way to define the entropy for measure, as we want to deal with possibly infinite-to-one maps. }

In Section \ref{sec-proofgcwp2} we prove Theorem \ref{th-GCWP}. The main trick is the Hubbard-Stratonovich formula and then the Laplace method, as in \cite{Leplaideur-Watbled1}. 

In Section \ref{sec-vlasov} we discuss an application to the mean-field $XY$ model. 


\section{Proof of Theorem \ref{th-rayonana}}\label{sec-proofthrayo}

\subsection{Properties for the Transfer Operator with infinite alphabet}

\subsubsection{First  spectral properties: $\CL_{\phi}$ is quasi-compact}
The function $A$ is continuous thus uniformly continuous (since $E\times E$ is compact) and with values in $\{0,1\}$. Therefore, there exists $\eps_{A}\in\, ]0,1[$ such that for any $u$, $u'$, $t$ and $t$' in $E$ satisfying $d(u,u')<\eps_{A}$ and $d(t,t')<\eps_{A}$, 
$$A(t,u)=A(t',u').$$

\begin{lemma}
\label{opagibien}
$\CL_{\phi}$ acts on $\CC^{0}(\Om)$ and on $\Lipom$. 
\end{lemma}
\begin{proof}
Let $f$ be in $\CC^{0}(\Om)$.
The function $x\mapsto e^{\phi(tx)}A(t,x_{0})f(tx)$ is continuous on $\Om$ 
and 
$$|e^{\phi(tx)}A(t,x_{0})f(tx)|\leq e^{\|\phi\|_\infty}\|f\|_\infty,$$
thus by the dominated convergence theorem $\CL_{\phi}(f)$ is continuous on $\Om$. 
Moreover $\CL_{\phi}$ acts continuously on $\CC^{0}(\Om)$ with operator norm
$$\|\CL_{\phi}\|_\infty\leq e^{\|\phi\|_\infty}.$$
Now let $f$ be in $\Lipom$. Notice that for any $t$ in $E$ and $x$, $y$ in $\Om$, 
$$d_\Om(tx,ty)=\frac{1}{2}d_\Om(x,y),$$
so that the function $f_t:x\mapsto f(tx)$ is Lipschitz with $\Lip(f_t)\leq \frac{1}{2}\Lip(f)$.
It is easy to show that if $\phi$ is Lipschitz, then $e^\phi$ is Lipschitz with 
$$\Lip(e^\phi)\leq e^{\|\phi\|_\infty}\Lip(\phi).$$
Notice also that for any $t$ in $E$, the map $A_t:x\mapsto A(t,x_0)$ is Lipschitz with 
$$\Lip(A_t)\leq 2\Lip(A).$$
As the product of two Lipschitz functions $f$ and $g$ is Lipschitz 
with 
$$\Lip(fg)\leq \|f\|_\infty \Lip (g) + \|g\|_\infty \Lip (f),$$
we easily deduce that $\cal L_\phi(f)$ is Lipschitz and that $\CL_{\phi}$ acts continuously on $\Lipom$.  
\end{proof}

\begin{lemma}
\label{lem-spectralradiusope}
The spectral radius  on $\CC^{0}(\Om)$ satisfies $\log r_{\phi}=\disp\lim_{\ninf}\frac1n\log||\CL_{\phi}^{n}(\BBone)||_{\8}$.
\end{lemma}
\begin{proof}
We recall that $r_{\phi}:=\disp\lim_{n\to +\infty}\|\CL^{n}_{\phi}\|_\infty^{1/n}
=\inf_{n\geq 1}\|\CL^{n}_{\phi}\|_\infty^{1/n}$. 
For any $f$ in $\CC^{0}(\Om)$, $x\in\Om$,
$$|\CL_{\phi}(f)(x)|\leq\int_E e^{\phi(tx)}A(t,x_{0})|f(tx)|\,d\rho(t)
\leq \|f\|_\infty \CL_{\phi}(\BBone)(x)\leq \|f\|_\infty \|\CL_{\phi}(\BBone)\|_\infty,$$
hence
$\|\CL_{\phi}\|_\infty= \|\CL_{\phi}(\BBone)\|_\infty$.
For $n\in\bb N$,
$$\CL_{\phi}^n(f)(x)=\int_{\Om_{n}(x_0)} e^{S_n(\phi)(tx)}f(tx)\rho^{\otimes n}(dt),$$
where $t$ inside the integral stands for $t=t_0\cdots t_{n-1}$
and $\rho^{\otimes n}(dt)=\prod _{i=0}^{n-1}\rho(dt_i)$.
Then $\|\CL_{\phi}^n\|_\infty= \|\CL_{\phi}^n(\BBone)\|_\infty$ for the same reason 
as for $n=1$, hence
$r_{\phi}=\disp\lim_{n\to +\infty}\|\CL^{n}_{\phi}(\BBone)\|_\infty^{1/n}$.
\end{proof}

As the measure $\rho$ is of full support and $A$ satisfies the hypothesis (A3) it is easy
to show that for any $x$ in $\Om$,
$\CL_{\phi}(\BBone)(x)$ is strictly positive.
One then has that $$ \disp\wh\CL^{*}_{\phi}(\mu):=\frac{\CL^{*}_{\phi}(\mu)}{\CL^{*}_{\phi}(\mu)(\BBone)}$$
is a probability measure for any $\mu\in M(\Om)$.
The map $\disp\wh\CL^{*}_{\phi}:M(\Om)\to M(\Om)$ is continuous on the compact (for the weak*-topology) convex space $M(\Om)$ therefore by the Schauder-Tychonoff theorem, there exists a probability measure $\nu_{\phi}$ such that 
$$\wh\CL_{\phi}^{*}(\nu_{\phi})=\nu_{\phi}.$$
This measure is either called the \emph{conformal measure} or the \emph{eigen-measure}. 
In the following we set $\l_{\phi}:=\CL^{*}_{\phi}(\nu_\phi)(\BBone)=\disp\int\CL_{\phi}(\BBone)\,d\nu_{\phi}$. 
\begin{proposition}
\label{prop-rlambda}
With previous notations $\disp r_{\phi}=\l_{\phi}$. 
Moreover for any $x\in\Om$, 
\begin{equation}\label{limclnr}
\log r_{\phi}=\lim_{\ninf}\frac1n\log\CL^{n}_{\phi}(\1)(x).
\end{equation}
\end{proposition}
\begin{proof}
Lipschitz regularity for $\phi$ yields that the Bowen condition holds: for any $n$, for any $x$ and $y$ satisfying $x_{i}=y_{i}$ for $i=0,\ldots n-1$,
\begin{equation}
\label{equ-bowen0}
|S_{n}(\phi)(x)-S_{n}(\phi)(y)|\leq D\,\Lip(\phi),
\end{equation}
where $D=\Diam(\Om)=\max\{\dom(x,y); x,y\in\Om\}$.
Indeed if $x_{i}=y_{i}$ for $i=0,\ldots n-1$ then $\dom(\sigma^k x,\sigma^k y)=2^k \dom(x,y)$
for $0\leq k\leq n$, so that
\begin{multline*}|S_{n}(\phi)(x)-S_{n}(\phi)(y)|\leq \sum_{k=0}^{n-1}\Lip(\phi) 2^k \dom(x,y)
\leq 2^n \Lip(\phi)d_\Om(x,y)\\
=\Lip(\phi)d_\Om(\sigma^n(x),\sigma^n(y)).
\end{multline*}
\begin{lemma}\label{lem-prop-rlambda}
For any $x,y$ in $\Om$ such that $\dom(x,y)<\frac{\eps_{A}}{2}$, for any $n\geq 2$,
\begin{equation}
\label{equ-bowen2}
e^{-D\Lip(\phi)}\le \frac{\CL_{\phi}^{n}(\BBone)(x)}{\CL_{\phi}^{n}(\BBone)(y)}\le e^{D\Lip(\phi)}.
\end{equation} 
\end{lemma}
\begin{proof}
If $\dom(x,y)<\frac{\eps_{A}}{2}$ then $d(x_0,y_0)<\eps_A$ so that $\Om_n(x_0)=\Om_n(y_0)$.
For any $a\in\Om_n(x_0)$, \eqref{equ-bowen0} implies that
$$|S_n(\phi)(ax)-S_n(\phi)(ay)|\leq D\,\Lip(\phi),$$
therefore
$$e^{S_n(\phi)(ay)}e^{-D\Lip(\phi)}\leq e^{S_n(\phi)(ax)}\leq e^{S_n(\phi)(ay)}
e^{D\Lip(\phi)}$$
and by integrating over $\Om_n(x_0)$ one gets
$$\CL_{\phi}^n(\BBone)(y)e^{-D\Lip(\phi)}\leq \CL_{\phi}^n(\BBone)(x)\leq 
\CL_{\phi}^n(\BBone)(y)e^{D\Lip(\phi)}.$$
\end{proof}

We pick $N_{1}$ sufficiently big such that Assumption (A3) holds, that is 
$$\forall a, b \in E,\ \Om_{N_1-1}(a,b)\neq\emptyset.$$

Then, we choose $N_{2}$ sufficiently big such that $2^{-N_{2}}<\ds\frac{\eps_{A}}{2\Diam(\Om)}$. 

\begin{claim}
There exists $C=C(\phi)>0$ such that for every $n>N_{1}+N_{2}$, for every $x$ and $y$ in $\Om$, 
\begin{equation}
\label{equ-bowen2-2}
e^{-C}\le \frac{\CL_{\phi}^{n}(\BBone)(x)}{\CL_{\phi}^{n}(\BBone)(y)}\le e^{C}.
\end{equation} 
\end{claim}
\begin{proof}[Proof of the Claim]
We pick $x$ and $y$ in $\Om$. We denote by $\mbf t$ an element $(t_{1},\ldots, t_{N_{2}})$ of $\Om_{N_{2}}$ and by $\mbf u$ and $\mbf v$ some elements of $\Om_{N_{1}}$. We set $N:=N_{2}+N_{1}$ and $m:=n-N$. 
$$\begin{aligned}
\CL_{\phi}^{n}&(\BBone)(x)=\CL_{\phi}^{N}\circ \CL_{\phi}^{m}(\BBone)(x)\\
&=\iint e^{S_{N}(\phi)(\mbf t\mbf u x)}A(t_{N_2},u_1)A(u_{N_1},x_0)\CL_{\phi}^{m}(\BBone)(\mbf t\mbf u x)\,d\rho^{\otimes N_{2}}(\mbf t)d\rho^{\otimes N_{1}}(\mbf u)\\
&=\iint e^{S_{N_2}(\phi)(\mbf t\mbf u x)}
e^{S_{N_1}(\phi)(\mbf u x)}A(t_{N_2},u_1)A(u_{N_1},x_0)\CL_{\phi}^{m}(\BBone)(\mbf t\mbf u x)\,d\rho^{\otimes N_{2}}(\mbf t)d\rho^{\otimes N_{1}}(\mbf u)
\end{aligned}$$
where we used the identity
$$S_{N_1+N_2}(\phi)=S_{N_2}(\phi)+S_{N_1}(\phi)\circ\sigma^{N_2}.$$
Let us set $m_A=\inf_{t\in E}\rho\pare{B(t,\eps_A)}$, which is positive thanks to hypothesis \textbf{(H)}.
As $\Om_{N_1-1}(u_1,y_0)\neq\emptyset$ we can pick $\underline{u}$ in it.
If $\mbf v\in\Om_{N_1}$ is such that $d(v_1,u_1)<\eps_A$ and $d(v_i,\underline{u}_{i-1})<\eps_A$ for every $2\leq i\leq N_1$, then 
$A(v_{N_1},y_0)=1$, that is $\mbf v\in\Om_{N_1}(y_0)$.
We deduce that 
$$\int_{\Om_{N_1}}A(v_{N_1},y_0)\BBone_{B(u_1,\eps_A)}(v_1)d\rho^{\otimes N_{1}}(\mbf v)\geq m_A^{N_1}.$$
Therefore 
\begin{multline*}
\CL_{\phi}^{n}(\BBone)(x)\leq \frac1{m_{A}^{N_{1}}}
\iiint e^{S_{N_2}(\phi)(\mbf t\mbf u x)}
e^{S_{N_1}(\phi)(\mbf u x)}A(t_{N_2},u_1)A(u_{N_1},x_0)\CL_{\phi}^{m}(\BBone)(\mbf t\mbf u x)\\
A(v_{N_1},y_0)\BBone_{B(u_1,\eps_A)}(v_1)d\rho^{\otimes N_{1}}(\mbf v)
d\rho^{\otimes N_{2}}(\mbf t)d\rho^{\otimes N_{1}}(\mbf u).
\end{multline*}
From \eqref{equ-bowen0} we deduce that 
$$e^{S_{N_2}(\phi)(\mbf t\mbf u x)}\leq e^{D\Lip(\phi)}e^{S_{N_2}(\phi)(\mbf t\mbf v y)}.$$
As $d_\Om(\mbf t\mbf u x,\mbf t\mbf v y)=2^{-N_2}d_\Om(\mbf u x,\mbf v y)<\dfrac{\eps_A}{2}$
we know from \eqref{equ-bowen2} that
$$\CL_{\phi}^{m}(\BBone)(\mbf t\mbf u x)
\leq e^{D\Lip(\phi)}\CL_{\phi}^{m}(\BBone)(\mbf t\mbf v y).$$
As $e^{S_{N_1}(\phi)(\mbf u x)}\leq e^{2N_1\|\phi\|_\infty }e^{S_{N_1}(\phi)(\mbf v y)}$
we deduce eventually that
\begin{eqnarray*}
\CL_{\phi}^{n}(\BBone)(x)&\le& \frac{e^{2D\Lip(\phi)+2N_1\|\phi\|_\infty}}{m_{A}^{N_{1}}}\iiint e^{S_{N_{2}}(\phi)(\mbf t\mbf v y)}
e^{S_{N_1}(\phi)(\mbf v y)}\CL_{\phi}^{m}(\BBone)(\mbf t\mbf v y)A(t_{N_2},u_1)\\
&&\hskip 1cm A(u_{N_1},x_0)A(v_{N_1},y_0)\BBone_{B(u_1,\eps_A)}(v_1)\,d\rho^{\otimes N_{1}}(\mbf v)\,d\rho^{\otimes N_{2}}(\mbf t)d\rho^{\otimes N_{1}}(\mbf u)\\
&\le& \frac{e^{2D\Lip(\phi)+2N_1\|\phi\|_\infty}}{m_{A}^{N_{1}}}\iint e^{S_{N}(\phi)(\mbf t\mbf v y)}\CL_{\phi}^{m}(\BBone)(\mbf t\mbf v y)\\
&&\hskip 5cm A(t_{N_2},v_1)A(v_{N_1},y_0)\,d\rho^{\otimes N_{2}}(\mbf t)d\rho^{\otimes N_{1}}(\mbf v)\\
&&\text{ where we use }A(t_{N_2},v_{1})=A(t_{N_2},u_{1}),\\
&\le &\frac{e^{2D\Lip(\phi)+2N_1\|\phi\|_\infty}}{m_{A}^{N_{1}}}\CL^{n}_{\phi}(\BBone)(y).
\end{eqnarray*}
Exchanging $x$ and $y$  we get the reverse inequality. 
\end{proof}
We can now finish the proof of Proposition \ref{prop-rlambda}. 
First we recall that according to Lemma \ref{lem-spectralradiusope},
$\log r_{\phi}=\disp\lim_{\ninf}\frac1n\log||\CL_{\phi}^{n}(\BBone)||_{\8}$.
As $\Om$ is compact there exists $x_n\in\Om$ such that $\|\CL_{\phi}^{n}(\BBone)||_{\8}=\CL_{\phi}^{n}(\BBone)(x_n)$.
For any $x$ in $\Om$ and $n>N_1+N_2$ we get from \eqref{equ-bowen2-2} that
\begin{equation}\label{equ-bowen2-3}
e^{-C}\CL_{\phi}^{n}(\BBone)(x_n)\leq  \CL_{\phi}^{n}(\BBone)(x) \leq e^C \CL_{\phi}^{n}(\BBone)(x_n),
\end{equation}
therefore
$$-\frac{C}{n}+\frac{1}{n}\log \CL_{\phi}^{n}(\BBone)(x_n)\leq \frac{1}{n}\log \CL_{\phi}^{n}(\BBone)(x) \leq \frac{C}{n}+\frac{1}{n}\log \CL_{\phi}^{n}(\BBone)(x_n),$$
and taking the limit we get \eqref{limclnr}.
Now integrating \eqref{equ-bowen2-3} we get
$$e^{-C}\CL^{n}_{\phi}(\BBone)(x_n)\le \int \CL^{n}_{\phi}(\BBone)(x)\,d\nu_{\phi}(x)\le e^{C}\CL^{n}_{\phi}(\BBone)(x_n),$$
and since $\l^{n}_{\phi}=\disp\int \CL^{n}_{\phi}(\BBone)(x)\,d\nu_{\phi}(x)$
we get $\log\l_{\phi}=\log r_{\phi}$. 
\end{proof}

We claim that we can apply the Ionescu-Tulcea \& Marinescu  Theorem\footnote{ITM Theorem in short.} (see \cite{ITM}, see also \cite{broise}, Theorem 4.2 or
\cite{Norman}, Theorem 2.1) to get a spectral decomposition of 
the operator $\disp\wt\CL_{\phi}:=\frac1{r_{\phi}}\CL_{\phi}$. 
Indeed the spaces $\CC^{0}(\Om)$, $\CC^{+1}(\Om)$ satisfy the first hypothesis of ITM Theorem, which is
\begin{enumerate}
\item if $f_n\in \CC^{+1}(\Om)$, $f\in \CC^{0}(\Om)$, $\Lim{n}{\infty}\|f_n-f\|_\infty=0$,
and $\|f_n\|_L\leq C$ for all $n$, then $f\in \CC^{+1}(\Om)$ and $\|f\|_L\leq C$,
\end{enumerate}
and $\disp\wt\CL_{\phi}$ satisfies the three following hypothesis:
\begin{enumerate}[resume]
\item $\sup_{n\in\bb N} \{\|\disp\wt\CL_{\phi}^n(f)\|_\infty,\,f\in \CC^{+1}(\Om),\,\|f\|_L\leq 1\}<+\infty$,
\item there exists $a\in ]0,1[$, $b>0$ and $n_{0}\geq 1$ such that for any $f\in \Lipom$,
$$\|\disp\wt\CL_{\phi}^{n_0}(f)\|_L\leq a\|f\|_L+b\|f\|_\infty,$$
\item if $V$ is a bounded subset of $(\Lipom,\|\cdot\|_L)$, then $\disp\wt\CL_{\phi}^{n_0}(V)$
has compact closure in $(\CC^{0}(\Om),\|\cdot\|_\infty)$.
\end{enumerate}

We sketch the proof of (3) and let the reader check the other conditions.

\bigskip

\textit{Proof of (3).}
A direct computation yields that for $f$ Lipschitz continuous
\begin{multline*}
\left|\CL_{\phi}^{n}(f)(x)-\CL_{\phi}^{n}(f)(y)\right|\leq
\Lip(f)\frac{\dom(x,y)}{2^{n}}\CL_{\phi}^{n}(\BBone)(x)\\
+e^{n\|\phi\|_\infty }\|f\|_\infty(\Lip(\phi)+2\Lip(A))\dom(x,y).
\end{multline*}
From \eqref{equ-bowen2-2} we know that for any $n>N_1+N_2$, for any $x$ in $\Om$,
$$e^{-C}\leq \frac{\CL_{\phi}^{n}(\BBone)(x)}{\lambda_\phi^n}\leq e^C,$$
hence as $r_\phi=\lambda_\phi$ we have
\begin{multline*}
\left|\wt\CL_{\phi}^{n}(f)(x)-\wt\CL_{\phi}^{n}(f)(y)\right|\leq
\Lip(f)\frac{\dom(x,y)}{2^{n}}e^C\\
+\frac{e^{n\|\phi\|_\infty }}{r_\phi^n}\|f\|_\infty(\Lip(\phi)+2\Lip(A))\dom(x,y).
\end{multline*}
Therefore
$$\|\wt\CL_{\phi}^{n}(f)\|_L=\Lip(\wt\CL_{\phi}^{n}(f))+\|\wt\CL_{\phi}^{n}(f)\|_\infty
\leq A_n\Lip(f)+B_n\|f\|_\infty\leq  A_n\|f\|_L+B_n\|f\|_\infty$$
where $A_n=\dfrac{e^C}{2^n}$
and $B_{n}=\dfrac{e^{n\|\phi\|_\infty }}{r_\phi^n}(\Lip(\phi)+2\Lip(A)+1)$. 
Picking any $a$ in $]0,1[$ and adjusting $n$ such that $2^{-n}e^{C}<a$ one gets the result. \hfill\qed

In particular, and considering $\CL_{\phi}$ as an operator on $\Lipom$, the proof of the ITM Theorem  shows (see \cite[Lem. 4.7]{broise}, or \cite[Lemma 2.4]{Norman}) that  $r_{\phi}$ is an eigenvalue for $\CL_{\phi}$ associated to the function in $\Lipom$ defined by 
$$G_{\phi}:=\lim_{\ninf}\frac1n\sum_{k=0}^{n-1}\wt\CL_{\phi}^{k}(\BBone).$$
Furthermore, we have the following decomposition 
$$\wt\CL_{\phi}:=\sum e^{i\theta_{j}}\Pi_{j}+\Psi$$
where the $\Pi_{j}$'s are (finitely many) projectors with finite rank, the $\theta_{j}$'s are real numbers and $\Psi$ has spectral radius strictly smaller than 1. Moreover, 
$$\Pi_{k}\Pi_{j}=0\text{ if }j\ne k\text{ and }\Psi\Pi_{j}=\Pi_{j}\Psi=0.$$ 

\subsubsection{Second decomposition of the spectrum: $r_{\phi}$ is 
the unique eigenvalue with ma\-xi\-mal modulus and its eigenspace is of dimension one}
For simplicity we set $\theta_{0}=0$. We shall see that mixing yields more precise results on the spectral decomposition of $\CL_{\phi}$. 

\begin{lemma}
\label{lem-preimdense}
For any $x\in\Om$, $\disp\bigcup_{n\ge 0}\s^{-n}(\{x\})$ is dense in $\Om$. 
\end{lemma}
\begin{proof}
Let $x$ and $y$ be in $\Om$, let $\eps>0$.
Let $n\in\bb N$ be such that $2^{-n}\Diam(\Omega)<\eps$, and let $z_i=y_i$ for every $i$ in $\interventier{0}{n-1}$, so that 
$d_\Om(y,z)\leq \eps$.
According to assumption \textbf{(A3)} there exist $u_n,\cdots,u_{n+N-2}$ in $E$ such that  
$z:=y_{0}\ldots y_{n-1}u_{n}\ldots u_{n+N-2}x$
belongs to $\Om$.
Then $z$ belongs to $\s^{-(n+N-1)}(\{x\})\cap B(y,\eps)$.
\end{proof}
\begin{remark}
\label{rem-superdensiteunif}
Actually, we have proved a better result: for any $\eps$, there exists $N'=N'(\eps)$ such that for any $y$ and $x$, $B(y,\eps)\cap \s^{-N'}(\{x\})\neq\emptyset$.
$\blacksquare$\end{remark}

\begin{proposition}
\label{prop-spectre}
The spectral radius  $r_{\phi}$ is a simple single dominating eigenvalue. 
The rest of the spectrum for $\CL_{\phi}$ is a compact set strictly inside the disk $\mathbb D(0,r_{\phi})$.
\end{proposition}
\begin{proof}
Because of the first result on the spectrum of $\CL_{\phi}$, it remains to prove that $r_{\phi}$ is simple and that any other eigenvalue has modulus strictly lower than $r_{\phi}$. For that we
 use spectral properties of positive operators exposed in \cite[chap. 1\& 2]{Kras64}.  
We claim that the set $K$ of non-negative Lipschitz functions is a \emph{solid and reproducing cone} . Solid means it has non-empty interior and reproducing means $$\Lipom=K-K.$$
It is easy to see that any positive Lipschitz function is in $\inte K$.

$\bullet$ {\bf Step one}. We prove that for any $f\not\equiv 0\in K$, there exists $p$ such that $\CL_{\phi}^{p}(f)$ belongs to $\inte K$. 

Let $y$ be such that $f(y)>0$. Let $\eps>0$ be such that $d_\Om(y,y')\leq \eps\Longrightarrow f(y')>0$. 
According to Remark \ref{rem-superdensiteunif}, there exists
$p\in\bb N$ such that for any $x\in \Om$, $B(y,\eps)\cap \s^{-p}(\{x\})\neq\emptyset$.
Then $\CL_{\phi}^{p}(f)$ is positive.
Indeed let $x\in\Om$, and $z$ in $B(y,\eps)\cap \s^{-p}(\{x\})$.
Let $\eta:=\min(\eps_A,\frac{\eps}{2})$.
The definition of $\eps_{A}$ yields that every $t$ in $\ds\prod_{i=0}^{p-1}B(z_i,\eta)$ belongs to $\Omega_p(x_0)$, therefore
$$\CL_{\phi}^{p}(f)(x)=\int_{\Om_{p}(x_0)} e^{S_p(\phi)(tx)}f(tx)\rho^{\otimes p}(dt)
\geq \int_{\prod_{i=0}^{p-1}B(z_i,\eta)} e^{S_p(\phi)(tx)}f(tx)\rho^{\otimes p}(dt).$$
But if $t\in \ds\prod_{i=0}^{p-1}B(z_i,\eta)$ then
$$d_\Om(tx,y)\leq d_\Om(tx,z)+d_\Om(z,y)\leq \eta+\frac{\eps}{2}\leq \eps$$
hence $f(tx)>0$.
As any non-empty ball in $E$ has positive $\rho$-measure we deduce that 
$\CL_{\phi}^{p}(f)(x)>0$.

$\bullet$ {\bf Step two}. End of the proof. 
We deduce from step one that $\CL_{\phi}$ is strongly positive (see \cite[Definitions 2.1.1]{Kras64}). Therefore it is $u$-positive for any $u\in \inte K$. From Th. 2.10, 2.11 and 2.13 we deduce that $r_{\phi}$ is a simple eigenvalue and that every other eigenvalue $\lambda$ of $\CL_{\phi}$ satisfies the inequality
$|\lambda|<r_\phi$.
\end{proof}

To re-employ notation from above, there is only one $\Pi_{0}$, no other $\Pi_{i}$'s. Furthermore, using the fact that $\nu_{\phi}$ is an eigenmeasure, one easily gets that  
for any $f\in \Lipom$,
\begin{equation}
\label{equ1-formCL}
\CL_{\phi}(f)=r_{\phi}\underbrace{\int f\,d\nu_{\phi}\cdot G_{\phi}}_{=\Pi_{0}(f)}+r_{\phi}\Psi(f).
\end{equation}

\subsection{Gibbs measure and ergodic properties}

\subsubsection{The Gibbs measure and its main properties}

Let $\mu_{\phi}$ be the measure defined by $d\mu_{\phi}:=G_{\phi}d\nu_{\phi}$. We emphasize that by construction $\mu_{\phi}$ is a probability measure. 
We shall use the following fact: for every $n\in\bb N$, 
\begin{equation}\label{fact}
\CL_{\phi}^n(f\cdot g\circ\s^n)=g\cdot \CL_{\phi}^n(f).
\end{equation}

\begin{lemma}
\label{lem-muphiinv}
The measure $\mu_{\phi}$ is $\s$-invariant. It is called the \emph{Dynamical Gibbs Measure (DGM in short)} associated to $\phi$. 
\end{lemma}
\begin{proof}
For $f$ continuous
\begin{eqnarray*}
\int f\circ\s\,d\mu_{\phi}&=& \int f\circ\s\cdot G_{\phi}\,d\nu_{\phi}\\
&=& \frac1{r_{\phi}}\int \CL_{\phi}(f\circ \s\cdot G_{\phi})\,d\nu_{\phi}\\
&=&  \frac1{r_{\phi}}\int f\cdot\CL_{\phi}( G_{\phi})\,d\nu_{\phi}=\int f\,d\mu_{\phi}.
\end{eqnarray*}
\end{proof}

\begin{proposition}
\label{prop-propergomuphi}
The measure $\mu_{\phi}$ is mixing thus ergodic. 
\end{proposition}
\begin{proof}
Let $f$ and $g$ be two functions in $\Lipom$. Then
\begin{eqnarray*}
\int f\cdot g\circ\s^{n}\,d\mu_{\phi}&=& \int f\cdot g\circ \s^{n}\cdot G_{\phi}\,d\nu_{\phi}\\
&=& \frac1{r_{\phi}^{n}}\int\CL_{\phi}^{n}(fG_{\phi}.g\circ \s^{n})\,d\nu_{\phi}\\
&=& \frac1{r_{\phi}^{n}}\int\CL_{\phi}^{n}(fG_{\phi}).g\,d\nu_{\phi}\\
&=&\int\left(\int fG_{\phi}\,d\nu_{\phi}\cdot G_{\phi}+\Psi^{n}(fG_{\phi})\right)\cdot g\,d\nu_{\phi}.
\end{eqnarray*}
We have seen that the spectral radius of $\Psi$ is strictly  lower than 1. Therefore $\disp \Psi^{n}(fG_{\phi})$ goes to 0 for the Lipschitz norm, thus for the continuous norm. This yields 
$$\int f\cdot g\circ\s^{n}\,d\mu_{\phi}\to_{\ninf}\int f\,d\mu_{\phi}\int g\,d\mu_{\phi},$$
and the proposition is proved. 
\end{proof}

\subsubsection{Furthermore properties}

\begin{lemma}
\label{lem-Gphipositive}
There exists $C(\phi)$ such that for every $x$, $e^{-C(\phi)}\le G_{\phi}(x)\le e^{C(\phi)}$.
\end{lemma}
\begin{proof}
By definition, $G_{\phi}\ge 0$. Let us prove by contradiction it is positive. Assume that $G_{\phi}(x)=0$. 
Then, $\CL_{\phi}(G_{\phi})=r_{\phi}G_{\phi}$ shows that $G_{\phi}(tx)=0$ for $\rho$-\ae $t$ in $E$ such that $A(t,x_{0})>0$. As $A$ and $G_\phi$ are continuous and $\rho$ has full support, this yields that for every $t$ such that $A(t,x_{0})>0$ $G_{\phi}(tx)=0$. In other words, for every $y$ in $\s^{-1}(\{x\})$, $G_{\phi}(y)=0$. 
By induction we deduce that for every $n\in\bb N$, 
for every $z$ in $\s^{-n}(\{x\})$, $G_{\phi}(z)=0$. Now, the set $\disp\cup_{n\ge 0}\s^{-n}(\{x\})$ is dense, and $G_{\phi}$ is continuous everywhere and null on a dense set. It is thus null everywhere which is impossible because $\int G_{\phi}\,d\nu_{\phi}=1$. 
 This shows that $G_{\phi}$ is positive, thus bounded from below by some constant of the form $e^{-C(\phi)}$. 
Furthermore, $\Om$ is compact and then $G_{\phi}$ is bounded from above. 
\end{proof}

Lemma \ref{lem-Gphipositive} immediately yields
\begin{corollary}
\label{cormunuequiv}
Both measures $\mu_{\phi}$ and $\nu_{\phi}$ are equivalent.
\end{corollary}

\subsubsection{Regularity of the spectral radius}

\begin{proposition}
\label{prop-convex}
The map ${ P}:\phi\mapsto \log r_{\phi}$ is convex on $\Lipom$.
\end{proposition}
\begin{proof}
Let us pick $\phi_1$, $\phi_2$ in $\Lipom$, and $\al\in[0,1]$. Set $\phi:=\al \phi_1+(1-\al)\phi_2$. 
For $n\in\bb N$, $x\in\Om$,
$$\begin{aligned}
\CL_{\phi}^n(\ind)(x)&=\int_{E^n} e^{S_n(\phi)(tx)}\ind_{\Om_{n}(x_0)}(t)\rho^{\otimes n}(dt)\\
&=\int_{E^n} e^{\al S_n(\phi_1)(tx)}\ind^\alpha_{\Om_{n}(x_0)}(t)
e^{(1-\al) S_n(\phi_2)(tx)}\ind^{1-\al}_{\Om_{n}(x_0)}(t)\rho^{\otimes n}(dt)\\
&\leq \pare{\int_{E^n} e^{S_n(\phi_1)(tx)}\ind_{\Om_{n}(x_0)}(t)\rho^{\otimes n}(dt)}^\al\pare{\int_{E^n} e^{S_n(\phi_2)(tx)}\ind_{\Om_{n}(x_0)}(t)\rho^{\otimes n}(dt)}^{1-\al}
\end{aligned}$$
therefore
$$\frac{1}{n}\log\left(\CL_{\phi}^n(\1)(x)\right)\le 
\al\frac{1}{n}\log\left(\CL_{\phi_1}^{n}(\1)(x)\right)+(1-\al)\frac{1}{n}\log\left(\CL^{n}_{\phi_2}(\1)(x)\right).$$
We deduce from \eqref{limclnr} that
$$\log r_{\al \phi_1+(1-\al)\phi_2}\leq \al \log r_{\phi_1}+(1-\al)\log r_{\phi_2},$$
which proves the convexity of ${ P}$.
\end{proof}

Let $\ogpsi$ be as above.
We recall the definition
$$I(\ogpsi):=\left\{\int \ogpsi\,d\mu,\ \mu\in M_\sigma(\Omega)\right\}.$$
By definition $I(\ogpsi)$ is a convex and closed set. 

\begin{proposition}
\label{prop-diff-convex}
The map ${\CP}:\gt\mapsto \log r_{\gt\cdot\ogpsi}$ is convex and 
infinitely differentiable on $\bb R^q$
with
\begin{equation}\label{nablaf}
\nabla \CP(\gt)=\int\ogpsi\,d\mu_{\gt\cdot\ogpsi}.
\end{equation}
For any $\gz=\nabla \CP(\gt)$ in $\nabla \CP(\bb R^q)$, 
$\CH(z,\ogpsi)$ is finite with
{\begin{equation}\label{Hnabla}\CH(\nabla \CP(\gt),\ogpsi)=
\CP(\gt)-\gt\cdot\nabla \CP(\gt)=\log r_{\gt\cdot\ogpsi}-\int \gt\cdot\ogpsi\,d\mu_{\gt\cdot\ogpsi}.
\end{equation}}
If $\gz$ does not belong to the closure $\ov{\nabla \CP(\bb R^q)}$ of 
$\nabla \CP(\bb R^q)$, in particular when $z\notin I(\ogpsi)$, then $\CH(\gz,\ogpsi)=-\infty$.
\end{proposition}

\begin{proof}
The convexity of $\CP$ follow from Proposition \ref{prop-convex}.
The map $Q$ with values in $\cal L(\Lipom)$
defined on $\bb R^q$ by 
$$Q(\gt)=\CL_{\gt\cdot\ogpsi}$$
is infinitely differentiable with 
$$\frac{\partial Q}{\partial t_k}(\gt)(g)=Q(\gt)(\psi_k g).$$
Adapting the proof of Thm. III.8 and Corollary III.11. of
\cite{hennion-herve} we see that the map $\gt\mapsto r_{\gt\cdot\ogpsi}$ is infinitely differentiable with
$$\frac{\partial r_{\gt\cdot\ogpsi} }{\partial t_k}(\gt)=r_{\gt\cdot\ogpsi}\int\psi_k\,d\mu_{\gt\cdot\ogpsi},$$
from which we deduce \eqref{nablaf}.

The conjugate function $\CP^*$ of $\CP$, defined by
$$\CP^*(\gz)=\sup_{\gt\in\bb R^q} (\gt\cdot\gz-\CP(\gt)),$$
is convex on $\bb R^q$ with values in $]-\infty,+\infty]$.
We refer for instance to \cite{Rockafellar}, section 26, for the theory of conjugates of convex functions.
In particular it is known that 
$$\nabla \CP(\bb R^q)\subset \Dom \CP^*\subset \ov{\nabla \CP(\bb R^q)},$$
where $\Dom \CP^*=\{\gz\in\bb R^q,\,\CP^*(\gz)<+\infty\}$,
with 
$$\CP^*(\nabla \CP(\gt))=\gt\cdot \nabla \CP(\gt)-\CP(\gt).$$
As $\CH(\cdot,\ogpsi)=-\CP^*$ the proof is finished.
\end{proof}

\section{Proof of Theorem \ref{thmpressquadra}} \label{sec-proofthcwp}

\subsection{Auxiliary functions \texorpdfstring{$\varphi$}{} and \texorpdfstring{$\fibar$}{}}
We recall the definitions of $\varphi_{\be}$ and $\fibar$ defined on $\bb R^{q}$: 
\begin{equation}\label{eqphi}
\varphi_{\beta}(\gt)=-\frac{\beta}{2}\|\gt\|^2+\log r_{\be\gramat t\cdot\ogpsi},\quad \fibar(\gz):=\disp\CH(\gz,\ogpsi)+\frac\be2\|\gz\|^{2}.
\end{equation}
\subsubsection{The function $\varphi_{\be}$}
\bigskip
\begin{lemma}
\label{lem-majovarphi}
For every $\beta >0$ and every $\gt$ satisfying $||\gt||>4\|\ogpsi\|_\infty$, 
$$\disp\varphi_{\be}(\gt)<\CH_{top}-\frac{\beta}{4}||\gt||^{2}.$$
\end{lemma}
\begin{proof}
Let us set $g(x):=\log r_{x\beta\gt\cdot\ogpsi}$ with $x\in[0,1]$. It is differentiable and Prop. \ref{prop-diff-convex} yields that for every $x$, 
$$g'(x)=\beta\gt\cdot\int \ogpsi\,d\mu_{x\beta\gt\cdot\ogpsi}.$$

Then, we use the mean value theorem. 
There exists $\theta\in]0,1[$ such that 
$$\log r_{\beta\gt\cdot\ogpsi}=g(1)=\CH_{top}+g'(\theta)=\CH_{top}+\beta\gt\cdot\int \ogpsi\,d\mu_{\theta\beta\gt\cdot\ogpsi}.$$
This  yields 
$$\varphi_{\be}(\gt)=\log r_{\beta\gt\cdot\ogpsi}-\frac\be2||\gt||^{2}\le \CH_{top}
+\beta\|\gt\|\|\ogpsi\|_\infty-\frac\be2||\gt||^{2}.$$
Now 
$$\beta\|\gt\|\|\ogpsi\|_\infty-\frac\be2||\gt||^{2}-(-\frac{\beta}{4}||\gt||^{2})
=-\frac{\beta\|t\|}{4}\pare{\|t\|-4\|\ogpsi\|_\infty}$$
and we get the result.
\end{proof}
We emphasize an immediate consequence of Lemma \ref{lem-majovarphi}:
all the maxima for $\varphi_{\be}$ are reached at critical points and inside the hypercube $[-K,K]^{q}$ if $K$ is chosen greater than $4\|\ogpsi\|_\infty$. 
Indeed $\varphi_{\be}(\textbf{0})=\CH_{top}$ and if $\gt$ is outside
the hypercube $[-K,K]^{q}$ with $K\geq 4\|\ogpsi\|_\infty$
then $||\gt||>4\|\ogpsi\|_\infty$, which implies
$\varphi_{\be}(\gt)<\CH_{top}$.

\subsubsection{The function $\fibar$}

We also recall the definition 
\begin{equation}
\label{equ1-logrhz}
\CH(\gz,\ogpsi):=\inf_{\gt\in\R^{q}}\left\{\log r_{\gt\cdot\ogpsi}-\gt.\gz\right\}=-\CP^\star(z),
\textrm{ where }\CP(\gt)=\log r_{\gt\cdot\ogpsi}.
\end{equation}
From the theory of conjugate functions we know that $\CH(\cdot,\ogpsi)$ is concave and upper semi-continous, with values
in $[-\infty,+\infty[$.

We emphasize that Proposition \ref{prop-diff-convex} yields that $\fibar=-\infty$ outside $\CI(\ogpsi)$, and $\fibar$ is finite on $\nabla \CP(\bb R^q)$.
Consequently all the maxima for $\fibar$ are reached inside the hypercube $[-\|\ogpsi\|_\infty,\|\ogpsi\|_\infty]^{q}$, which contains $\CI(\ogpsi)$.

Moreover, in our setting the equality $\CP^{**}=\CP$ holds true, hence we know that
\begin{equation}
\label{equ2-logrhz}
\log r_{\gt\cdot\ogpsi}=\CP(\gt)=\sup_{\gz}\{\CH(\gz,\ogpsi)+\gt\cdot\gz\}.
\end{equation}

Let us set 
$$\wt\CH(\gz):=\left\{\begin{aligned}&\sup_{\mu}\left\{\wh\CH(\mu),\,\int\ogpsi\,d\mu=\gz\right\}\textrm{ if }\gz\in I(\ogpsi),\\
&-\infty \textrm{ if }\gz\notin I(\ogpsi).\end{aligned}\right.$$ 
Then, 

\begin{proposition}
\label{prop-maxentropyequim}
For every $\gz$ in $\bb R^q$, 
$\wt\CH(\gz)=\CH(\gz,\ogpsi)$.
\end{proposition}
\begin{proof}
For any $\gt\in\bb R^q$ we have
$$\begin{aligned}\CP(\gt)&=\sup_\mu\pare{\disp\wh\CH(\mu)+\gt\cdot\int\ogpsi\,d\mu}\\
&=\sup_{\gz\in\bb R^q}\sup_{\mu, \int\ogpsi\,d\mu=z}\pare{\wh\CH(\mu)+\gt\cdot\gz}\\
&=\sup_{\gz\in\bb R^q}\pare{\wt\CH(\gz)+\gt\cdot\gz}.
\end{aligned}$$
In other words $(-\wt\CH)^{*}=\CP$.
It is easily seen that $\wt\CH$ is concave. By Theorem 12.2 of \cite{Rockafellar}
the biconjugate $(-\wt\CH)^{**}$ of $-\wt\CH$ is equal to its closure.
The following lemma shows that $-\wt\CH$ is closed convex therefore
$-\wt\CH=\CP^*$.
As by definition $\CP^*=-\CH(\cdot,\ogpsi)$, we deduce that $\wt\CH=\CH(\cdot,\ogpsi)$.
\end{proof}

\begin{lemma}
\label{lem-wthusc}
The function $\wt\CH$ is upper semi continuous on $\bb R^q$.
\end{lemma}
\begin{proof}
Let $\gz$ be fixed in $\bb R^q$, let $(\gz_{n})$ be a sequence in $\bb R^q$ converging to $\gz$.
If $\gz$ is not in $\CI(\ogpsi)$ then neither is $\gz_n$ for $n$ big enough hence
$$\limsup_{\ninf}\wt\CH(\gz_{n})=-\infty=\wt\CH(\gz).$$
Let us thus assume that $\gz$ is in $\CI(\ogpsi)$.
If only a finite number of $\gz_n's$ belong to $\CI(\ogpsi)$ then 
$$\limsup_{\ninf}\wt\CH(\gz_{n})=-\infty\leq \wt\CH(\gz).$$
If an infinite number of $\gz_n's$ belong to $\CI(\ogpsi)$ then to compute the limsup we can assume without loss of generality that every $\gz_n$ is in $\CI(\ogpsi)$.
Let $\mu_{n}$ be an invariant measure such that $\disp\int\ogpsi\,d\mu_{n}=\gz_{n}$ and 
\begin{equation}
\label{eq1-wthusc}
\wh\CH(\mu_{n})\ge \wt\CH(\gz_{n})-\frac1n.
\end{equation}
Let $\mu$ be any accumulation point for $(\mu_{n})$ for the weak* topology. For simplicity we shall write $\mu=\lim_{\ninf}\mu_{n}$. 

Then,  $\disp \int\ogpsi\,d\mu=\lim_{\ninf}\int\ogpsi\,d\mu_{n}=\lim_{\ninf}\gz_{n}=\gz$ and as the metric entropy is upper semi-continuous we get 
$$\wt\CH(\gz)\ge \wh\CH(\mu)\ge \limsup_{\ninf}\wh\CH(\mu_{n})\ge\limsup_{\ninf}\pare{\wt\CH(\gz_{n})-\frac1n}=\limsup_{\ninf}\wt\CH(\gz_{n}).$$
\end{proof}

Finally we have:
\begin{corollary}
\label{coro-goodfibar}
For every $\gz\in \bb R^q$, 
$\fibar(\gz)=\wt\CH(\gz)+\frac\be2||z||^{2}$. 
\end{corollary}

\subsubsection{Maxima for $\varphi_{\be}$ and $\fibar$}

The main result of this Subsection is 
\begin{proposition}
\label{prop-maxfifibar}
Inequality $\varphi_{\be}(\gz)\ge\fibar(\gz)$ holds for any $\gz$ in $\bb R^q$. Moreover $\varphi_{\be}(\gz)$ is maximum if and only if $\fibar(\gz)$ is maximum. Furthermore, if $\varphi_{\be}(\gz)$ is maximum then $\varphi_{\be}(\gz)=\fibar(\gz)$.
\end{proposition}

\begin{proof}
$\bullet$ Step 1. $\varphi_{\be}\ge \fibar$.
We use Equality \eqref{equ1-logrhz} with $\gt=\beta\gz$. This yields 

$$\fibar(\gz)=\disp\CH(\gz,\ogpsi)+\frac\be2\|\gz\|^{2}\leq \log r_{\gt\cdot\ogpsi}-\gt.\gz+\frac\be2||\gz||^{2}=\log r_{\be \gz\cdot\ogpsi}-\frac\be2||\gz||^{2}=\varphi_{\be}(\gz).$$

$\bullet$ Step 2. $\fibar(\gz)$ is maximal if and only if $\varphi_{\be}(\gz)$ is maximal and maximal values do coincide. 

Let $\gz$ be a maximum for $\varphi_{\be}$. Then, it is a critical point for $\varphi_{\be}$. 
As
$$\nabla \varphi_{\be}(\gz)=\beta \nabla \CP(\beta \gz)-\beta\gz$$
this yields $\gz=\nabla \CP(\beta \gz)$.
Using \eqref{Hnabla} we get
$$\CH(\gz,\ogpsi)=\CH(\nabla \CP(\beta \gz),\ogpsi)=\CP(\beta\gz)-\beta\gz\cdot\nabla \CP(\beta \gz)
=\CP(\beta\gz)-\beta\|\gz\|^2,$$
therefore 
$$\CP(\beta\gz)=\CH(\gz,\ogpsi)+\beta \|\gz\|^2.$$
Using step 1 and this last equality we get
$$\fibar(\gz)\leq \varphi_{\be}(\gz)=\CP(\beta\gz)-\frac{\beta}{2}\|\gz\|^2
=\CH(\gz,\ogpsi)+\beta \|\gz\|^2-\frac{\beta}{2}\|\gz\|^2=\fibar(\gz),$$
which shows that $\fibar(\gz)=\varphi_{\be}(\gz)$. 

On the other hand  for any $\gz'$,
$$\fibar(\gz')\le \varphi_{\be}(\gz')\le \varphi_{\be}(\gz)=\fibar(\gz),$$
which shows that $\gz$ is also a maximum for $\fibar$. 

Conversely, if $\gz$ is a maximum for $\fibar$, let $\gz'$ be any maximum for $\varphi_{\be}$. We get 
$$\fibar(\gz)\ge \fibar(\gz')=\varphi_{\be}(\gz')\ge \varphi_{\be}(\gz)\ge\fibar(\gz).$$
This shows that $\gz$ is also a maximum for $\varphi_{\be}$, which finishes the proof. 
\end{proof}

\begin{corollary}
\label{coro-maxfibar}
Maxima for $\fibar$ are reached on $\nabla\CP(\R^{q})$. 
\end{corollary}
\begin{proof}
Proposition \ref{prop-maxfifibar} states that maxima for $\fibar$ are maxima for $\varphi_{\be}$. We have seen after Lemma \ref{lem-majovarphi} that all the maxima for $\varphi_{\be}$ are reached at critical points. 

Now, $\gt$ is a critical point for $\varphi_{\be}(\gt)=-\frac{\beta}{2}\|\gt\|^2+\log r_{\be\gramat t\cdot\ogpsi}$ means 
$$\gt=\nabla\CP(\be\gt)\in \nabla\CP(\R^{q}).$$
\end{proof}

\subsection{Measures maximizing quadratic pressure}

\subsubsection{Another expression for $\CP_{2}(\be)$}

We remind that the metric entropy $\mu\mapsto \wh\CH(\mu)$ is upper semi-continuous (see Remark \ref{rem-whusc}).

Therefore the function 
$$F:\mu\mapsto { \wh\CH(\mu)}+\frac{\beta}{2}\left\|\int\ogpsi\,d\mu\right\|^2$$
is upper semicontinuous hence attains it supremum on the compact set
$M_\sigma(\Omega)$.
$$\begin{aligned}
\CP_2(\beta)&=\max_{\mu\in\ M_\sigma(\Omega)} F(\mu)\\
 &{=\max_{\gz\in \bb R^q}\max\left\{ \wh\CH(\mu)+\frac\be2||\gz||^{2},\ \int\ogpsi \,d\mu=\gz\right\}}\\
&=\max_{\gz\in \bb R^q}\pare{\wt\CH(\gz)+\frac{\beta}{2}\|\gz\|^2}\\
&=\max_{\gz\in \nabla\CP(\R^q)}\fibar(\gz)\\
&=\max_{\gz\in \bb R^q}\varphi_{\beta}(\gz)
\end{aligned}$$
where the last equality comes from Proposition \eqref{prop-maxfifibar}  {and the fourth equality comes from  Corollaries \ref {coro-goodfibar} and \ref{coro-maxfibar}}.

\subsubsection{Good DGM maximize quadratic pressure}

We note 
$$M:=\{\gz\in\bb R^q; \fibar(\gz) \textrm{ is maximal}\}.$$
Let $\gz\in M$. We saw in the proof of Proposition \eqref{prop-maxfifibar} that $\gz$ is then a critical point for $\varphi_\beta$
hence $\gz=\nabla \CP(\beta \gz)=\ds\int \ogpsi \, d\mu_{\beta\gz\cdot\ogpsi}$. 
From \eqref{unique} we know that
$$
\wh\CH(\mu_{\beta\gz\cdot\ogpsi})=\CP(\beta\gz)-\int\beta\gz\cdot\ogpsi\,d\mu_{\beta\gz\cdot\ogpsi}
$$
hence from \eqref{Hnabla} we deduce that
$\CH(\gz,\ogpsi)=\wh\CH(\mu_{\beta\gz\cdot\ogpsi}),$
thus
$$ \fibar(\gz)=F(\mu_{\beta\gz\cdot\ogpsi}).$$

Let $\mu$ be any measure, $\gz':=\int \ogpsi \, d\mu$.
Then 
$$F(\mu)=\wh\CH(\mu)+\frac\be2||\gz'||^{2}\leq \wt\CH(\gz')+\frac\be2\|\gz'\|^{2}=\fibar(\gz')\le \fibar(\gz).$$

Therefore,
$\mu_{\beta\gz\cdot\ogpsi}$ maximizes $F$.

\bigskip

\subsubsection{Maxima for quadratic pressure are realized only by good DGM}

Conversely let $\mu$ maximizing the function $F$. Set $\gz:=\ds\int\ogpsi\,d\mu$. Then
$\gz$ is in $M$ hence satisfies 
$$\gz=\ds\int \ogpsi \, d\mu_{\beta\gz\cdot\ogpsi}=\ds\int\ogpsi\,d\mu,$$
and $\wh\CH(\mu)=\wt\CH(\gz)$. Now using \eqref{unique} we can write
$$\CP(\be\gz)=\wh\CH(\mu_{\be\gz\cdot\ogpsi})+\be\gz\cdot\int\ogpsi\,d\mu_{\be\gz\cdot\ogpsi}=\wt\CH(\gz)+\be\gz\cdot\gz=\wh\CH(\mu)+\be\gz\cdot\int\ogpsi\,d\mu,$$
 which means that $\mu$ is equal to $\mu_{\be\gz\cdot\ogpsi}$ by uniqueness of the (linear) equilibrium state.

\section{Proof of Theorem \ref{th-GCWP}}\label{sec-proofgcwp2}
\subsection{A useful computation}

Let $f:\Omega\to\R$ be continuous. We want to evaluate the limit of 
$\disp \int f(\om)d\munbe(\om)$ as $n\to+\8$. In the first step we do the computation without the normalizing term $Z_{n,\be}$ and estimate it in the second step. 
We recall the identity

\begin{equation}
\label{equ-hubstrato}
e^{\|\gxi\|^2}=\frac{1}{(2\pi)^{q/2}}\int_{\bb R^q} \exp\pare{-\frac{1}{2}\|\gt\|^2+\sqrt 2\gt.\gxi}\,d\gt.
\end{equation}
Then we have 
$$\begin{aligned}
Z_{n,\be}\int_{\Omega} f(\om)&d\munbe(\om)=\int_{\Omega} e^{\frac\be{2n}||S_{n}(\ogpsi)(\om)||^{2}}f(\om)\,d\P(\om)\\
&= \frac{1}{(2\pi)^{q/2}}\int_{\Omega}\int_{\bb R^q} e^{-\frac{1}{2}\|\gt\|^2}
e^{\sqrt{\frac{\beta}{n}}\gt.S_{n}(\ogpsi)(\omega)}f(\om)\,d\gt\,d\P(\om)\\
&=\frac{1}{(2\pi)^{q/2}}\int_{\bb R^q} e^{-\frac{1}{2}\|\gt\|^2}\int_{\Omega}\int_{\Omega_n(\omega_0)}e^{\sqrt{\frac{\beta}{n}}\gt.S_{n}(\ogpsi)(\al\omega)}f(\al\om)\,d\rho^{\otimes n}(\al)\,d\P(\om)\,d\gt\\
&= \frac{1}{(2\pi)^{q/2}}\int_{\bb R^q} e^{-\frac{1}{2}\|\gt\|^2}\int_{\Omega}\CL^{n}_{\sqrt\frac{\be}{n}\gt\cdot\ogpsi}(f)(\om)
\,d\P(\om)\,d\gt \\
&= \pare{\frac{\beta n}{2\pi}}^{q/2}\int_{\bb R^q} 
e^{-\frac{n\beta}{2}\|\gz\|^2}\int_{\Omega}\CL^{n}_{\beta\gz\cdot\ogpsi}(f)(\om)
\,d\P(\om)\,d\gz 
\end{aligned}$$
where we made the change of variable $\beta \gz=\sqrt{\frac{\beta}{n}}\gt$ to get the last equality.

We claim that the part of the integral in $\gz$ outside the hypercube $[-K,K]^{q}$ is negligible with respect to the other part. 
Indeed, 
\begin{multline}\label{mult}
\int_{\R^{q}\setminus[-K,K]^{q}} 
e^{-\frac{n\beta}{2}\|\gz\|^2}\int_{\Omega}\CL^{n}_{\beta\gz\cdot\ogpsi}(f)(\om)
\,d\P(\om)\,d\gz \\
\leq 
\int_{\R^{q}\setminus[-K,K]^{q}} 
e^{-\frac{n\beta}{2}\|\gz\|^2}\|\CL^{n}_{\beta\gz\cdot\ogpsi}\|_\infty \|f\|_\infty\,d\gz, 
\end{multline}
and $\|\CL^{n}_{\beta\gz\cdot\ogpsi}\|_\infty\leq e^{n\beta\|\gz\|\|\ogpsi\|_\infty}$,
so that
$$e^{-\frac{n\beta}{2}\|\gz\|^2}\|\CL^{n}_{\beta\gz\cdot\ogpsi}\|_\infty\leq
e^{n\beta(\|\gz\|\|\ogpsi\|_\infty-\frac{1}{2}\|\gz\|^2)}
\leq e^{-n\beta\frac{\|\gz\|^2}{4}}$$
for $\|\gz\|>4\|\ogpsi\|_\infty$, as we noticed in the proof of
Lemma \eqref{lem-majovarphi}.
Now 
\begin{multline*}
\int_{\R^{q}\setminus[-K,K]^{q}}
e^{-n\beta\frac{\|\gz\|^2}{4}}\,d\gz
\leq \sum_{i=1}^q\int_{|z_i|>K}e^{-n\beta\frac{z_i^2}{4}}
\prod_{j\neq i} e^{-n\beta\frac{z_j^2}{4}}\,d\gz\\
=
\pare{\frac{4\pi}{n\beta}}^{\frac{q-1}{2}}
\sum_{i=1}^q\int_{|z_i|>K}e^{-n\beta\frac{z_i^2}{4}}
\,dz_i,
\end{multline*}
and 
$$\int_{|z_i|>K}e^{-n\beta\frac{z_i^2}{4}}
\,dz_i\leq \frac{4}{n\beta K}e^{-n\beta\frac{K^2}{4}}.$$
Returning to \eqref{mult} we get
\begin{equation}
\int_{\R^{q}\setminus[-K,K]^{q}} 
e^{-\frac{n\beta}{2}\|\gz\|^2}\int_{\Omega}\CL^{n}_{\beta\gz\cdot\ogpsi}(f)(\om)
\,d\P(\om)\,d\gz=O\pare{\frac{e^{-n\beta\frac{K^2}{4}}}{n^{\frac{q+1}{2}}}}.
\end{equation}
if $K$ is greater than $4\|\ogpsi\|_\infty$.

Now, we recall that $$\varphi_{\beta}(\gz)=-\frac{\beta}{2}\|\gz\|^2+\log r_{\be\gramat z.\ogpsi}$$
and that if $f$ belongs to $\CC^{+1}(\Om)$ then
$$\CL^{n}_{\beta\gz\cdot\ogpsi}(f)(\om)=
e^{n\log r_{\be\gz\cdot\ogpsi}}
\left[\pare{\int_{\Omega} f\,d\nu_{\beta\gz\cdot\ogpsi}}G_{\beta\gz\cdot\ogpsi}(\om)+\Psi^n_{\beta\gz\cdot\ogpsi}(f)(\omega)\right],$$
where the operator norm of $\Psi_{\beta\gz\cdot\ogpsi}$ acting on $\CC^{+1}(\Om)$
is strictly less than one.
We write $\Psi_{\beta\gz\cdot\ogpsi}=e^{-\eps(\beta,\gz)}T(\beta,\gz)$ 
where $\eps(\be,\gz)$ is the spectral gap of the operator $\CL_{\beta\gz\cdot\ogpsi}$ and $||T(\beta,\gz)||_{L}= 1$.
Then
\begin{multline}\label{multK}
\int_{[-K,K]^{q}} 
e^{-\frac{n\beta}{2}\|\gz\|^2}\int_{\Omega}\CL^{n}_{\beta\gz\cdot\ogpsi}(f)(\om)
\,d\P(\om)\,d\gz \\
=
\int_{[-K,K]^{q}} e^{n\varphi_\beta(z)}\int_\Omega\left[\pare{\int_{\Omega} f\,d\nu_{\beta\gz\cdot\ogpsi}}G_{\beta\gz\cdot\ogpsi}(\om)+e^{-n\eps(\beta,\gz)}T^n(\beta,\gz)(f)(\omega)\right]\,d\P(\om)
 \,d\gz
\end{multline}

The spectral gap $\eps(\be,\gz)$ is lower semi-continuous in $\gz$ hence it attains its infimum $m(\beta)$ on
the compact set $[-K,K]^{q}$, which is strictly positive.
We set 
$$\alpha(n,z,f)=\int_\Omega e^{-n\eps(\beta,\gz)}T^n(\beta,\gz)(f)(\omega)\,d\P(\om),$$
and notice that for any $\gz$ in $[-K,K]^{q}$,
$$|\alpha(n,z,f)|\leq e^{-nm(\beta)}\|f\|_L.$$
Eventually we get
\begin{multline}
\label{equ1-cvmunbe}
Z_{n,\beta}\int_\Omega f(\om)\,d\munbe(\om)=
\left(\frac{\be n}{2\pi}\right)^{q/2}
\int_{[-K,K]^{q}}e^{n\varphi_{\be}(\gz)}\\
\left[\pare{\int_{\Omega} f\,d\nu_{\beta\gz\cdot\ogpsi}}\pare{\int_{\Omega}G_{\beta\gz\cdot\ogpsi}\,d\P}+
\alpha(n,\gz,f)\right]\,d\gz
+O\pare{\frac{e^{-n\beta\frac{K^2}{4}}}{n^{\frac{q+1}{2}}}}
\end{multline}

The normalization term $Z_{n,\be}$ is obtained taking $f\equiv1$. This yields 
\begin{multline}
\label{equ4-cvmunbe}
\int_\Omega f(\om)\,d\munbe(\om)= \\
\frac{\ds\int_{[-K,K]^{q}} e^{n\varphi_{\be}(\gz)}
\left[\pare{\int_{\Omega} f\,d\nu_{\beta\gz\cdot\ogpsi}}\pare{\int_{\Omega}G_{\beta\gz\cdot\ogpsi}\,d\P}+
\alpha(n,\gz,f)\right]\,d\gz
+O\pare{\frac{e^{-n\beta\frac{K^2}{4}}}{n^{\frac{q+1}{2}}}}}
{\ds\int_{[-K,K]^{q}}e^{n\varphi_{\be}(\gz)}
\left[\int_{\Omega}G_{\beta\gz\cdot\ogpsi}\,d\P+
\alpha(n,\gz,\indic)\right]\,d\gz
+O\pare{\frac{e^{-n\beta\frac{K^2}{4}}}{n^{\frac{q+1}{2}}}}}.
\end{multline}

where $\alpha(n,\gz,f)$ and $\alpha(n,\gz,\indic)$ 
converge uniformly to $0$ with respect to $\gz$ when $n$ tends to infinity.

\subsection{The case \texorpdfstring{$q=1$}{}}
In this case the function $\varphi_\beta$ is analytic hence admits only finitely many maxima,
and we can argue as in \cite{Leplaideur-Watbled1}.

\subsection{The higher dimensional case}

We assume that all the maxima for $\fibe$ are non-degenerated. 
\begin{lemma}
\label{lem-finitemax}
The function $\varphi_{\be}$ only admits finitely many maxima. 
\end{lemma}
\begin{proof}
The proof is done by contradiction. Let us consider a sequence $(\gz_{n})$ of maxima for $\varphi_{\be}$. We have seen that  all the maxima are critical points and are in  some compact set $[-K,K]^{q}$ (see Lemma \ref{lem-majovarphi} and discussion after).

Therefore, we may consider some accumulation point $\gz$ for the $\gz_{n}$'s. For simplicity we set $\gz=\lim_{\ninf}\gz_{n}$ and we assume that $\gz_{n}\neq\gz_{n+1}$ holds for every $n$. Note that by continuity, $\gz$ is also a critical point for $\fibe$ and $\fibe$ is maximal at $\gz$

We remind that $\varphi_{\be}$ is $\CC^{\8}$. 
If we consider the restriction $\fibe_{n}$ of $\varphi_{\be}$ to each segment $[z_{n},z_{n+1}]$, then $\fibe_{n}$ is $\CC^{\8}$ and $\fibe'_{n}(\gz_{n})=\fibe'_{n}(\gz_{n+1})=0$. Hence, Rolle's theorem shows that there exists $\gz'_{n}\in[\gz_{n},\gz_{n+1}]$ such that \begin{equation}\label{eq-hessdege1}
\fibe''_{n}(\gz'_{n})=0.
\end{equation}

Set $\vec u_{n}:=\gz_{n}-\gz_{n+1}$ and consider any accumulation point $\vec u$ for $\disp \frac{\gz_{n}-\gz_{n+1}}{||\gz_{n}-\gz_{n+1}||}$. Equality \eqref{eq-hessdege1} can be rewritten under the form 
$$d^{2}\fibe(\gz'_{n})(\vec u_{n},\vec u_{n})=0,$$
which yields as $\ninf$ $d^{2}\fibe(\gz)(\vec u,\vec u)=0$. This means that $\gz$ is a degenerated maximal point for $\fibe$, which is in contradiction with our assumption.
\end{proof}

Let $\gz_1,\cdots,\gz_k$ be the points where $\varphi_\beta$ attains its maximum.
We recall that the Laplace method (see \cite[Ch.IX Th.3]{Wong}) states

$$\int_{0} e^{n\varphi_\be(z)}g(\gz)\,d\gz\sim_{n\to\infty} \frac{(2\pi)^{q/2} g(\gz_1)e^{n\varphi_\be(\gz_{1})}}{n^{q/2}\sqrt{|\det d^{2}\varphi_{\be}(\gz_1)|}},$$
provided that $\varphi_{\be}$ admits no other critical point than $\gz_{1}$  in an open set $O$ of $\bb R^q$, that $g(\gz_1)\neq 0$ and that the Hessian matrix $d^{2}\varphi_{\be}(\gz_1)$ is negative definite (which holds by our assumption). 

\begin{remark}
\label{rem-gnonnul}
We emphasize the assumption $g(\gz_{1})\neq 0$. 
$\blacksquare$\end{remark}

\medskip
We choose $K$ such that $\varphi_\be(\gz_{1})+\beta\frac{K^2}{4}>0$, and
letting $n\to+\infty$ in \eqref{equ4-cvmunbe}, we get that for every $f$ in $\CC^{+1}(\Om)$, 
$$\lim_{\ninf}\int_\Omega f(\om)\,d\munbe(\om)=\frac{\disp\sum_{j=1}^k\frac{\disp\int G_{\beta\gz_j\cdot\ogpsi}\,d\P}{\sqrt{\det d^{2}\varphi_{\be}(\gz_j)}}
\int f\,d\nu_{\beta\gz_j\cdot\ogpsi}}
{\disp\sum_{j=1}^k\frac{\disp\int G_{\beta\gz_j\cdot\ogpsi}\,d\P}{\sqrt{\det d^{2}\varphi_{\be}(\gz_j)}}},$$
which finishes the proof of Theorem \ref{th-GCWP}.

\section{Application to the mean-field \texorpdfstring{$XY$}{} model}\label{sec-vlasov}

\subsection{The cosine potential}

The mean-field $XY$ model is a system of $n$ globally coupled planar spins (or alternatively
of $n$ globally interacting particles constrained on a ring), with Hamiltonian
$$H_n=-\frac{1}{2n}\sum_{i,j=1}^{n}\cos(p_i-p_j),$$
where $p_i\in [0,\pi[$.
We can interpret it as a generalized Curie-Weiss-Potts model by setting
$E=\bb T=\{z\in\bb R^2,\|z\|=1\}$, $\Omega=\bb T^{\bb N}$,
and $\ogpsi(\om)=\om_0$.
Indeed every $\om_k$ in the word $\om=\om_0\om_1\cdots$ of $\Omega$ is uniquely expressed as $\om_k=(\cos\theta_k,\sin\theta_k)$ with $\theta_k$ in $[-\pi,\pi[$, and then
$$\|S_n(\ogpsi)(\omega)\|^2=\|\sum_{k=0}^{n-1}\om_k\|^2
=\sum_{i,j=0}^{n-1}\langle\om_i,\om_j\rangle
=\sum_{i,j=0}^{n-1}\cos(\om_i-\om_j).$$
We endow $\bb T$ with the usual distance on $\bb R^2$, and the Haar measure $\rho$ given by
 $$\int_{\bb T} h(z)\,\rho(dz)=\int_{-\pi}^{\pi} h(\vec u_\theta)\,\frac{d\theta}{2\pi}, \textrm{ where }\vec u_\theta=(\cos\theta,\sin\theta).$$
As $\ogpsi$ only depends on the first coordinate, we see that for any $\gt$ in $\bb R^2$ and any $f$ in $\CC^{0}(\Om)$,
$$\CL_{\gt\cdot\ogpsi}(f)(\om)=\int_{-\pi}^{\pi}
e^{\gt\cdot\vec u_\theta}\,f(\vec u_\theta\om)\,\frac{d\theta}{2\pi},$$
so that the spectral radius of $\CL_{\be\gt\cdot\ogpsi}$ is
$$r_{\be\gt\cdot\ogpsi}=\lambda_{\be\gt\cdot\ogpsi}=\int_{-\pi}^{\pi}
e^{\be\gt\cdot\vec u_\theta}\,\frac{d\theta}{2\pi},$$
with eigenfunction $G_{\be\gt\cdot\ogpsi}=\indic$, and 
$\nu_{\be\gt\cdot\ogpsi}=\mu_{\be\gt\cdot\ogpsi}$.
We notice that $r_0=1$.
If $\gt\neq 0$, we denote by $|\gt|$ its euclidean norm and by $\theta_\gt$ the unique element of $[-\pi,\pi[$ such that
$\gt=|\gt| \vec u_{\theta_{\gt}}$.
Then
$$\begin{aligned}
r_{\be\gt\cdot\ogpsi}&=\int_{-\pi}^{\pi}
e^{\be |\gt|\cos(\theta_{\gt}-\theta)}\,\frac{d\theta}{2\pi}\\
&=\int_{\theta_{\gt}-\pi}^{\theta_{\gt}+\pi}
e^{\be |\gt|\cos y}\,\frac{dy}{2\pi}\\
&=\int_{-\pi}^{\pi}
e^{\be |\gt|\cos y}\,\frac{dy}{2\pi}
\end{aligned}$$
because the integral does not depend on the interval of length $2\pi$ where we compute it.
Eventually we have
\begin{equation}\label{Bessel}
r_{\be\gt\cdot\ogpsi}=\int_{0}^{\pi}
e^{\be |\gt|\cos y}\,\frac{dy}{\pi}=I_0(\be |\gt|),
\end{equation}
where $I_0$ is the modified Bessel function of order zero,
and we get 
$$\varphi_{\beta}(\gt)=-\frac{\beta}{2}|\gt|^2+\log r_{\be\gt\cdot\ogpsi}=-\frac\be2 |\gt|^{2}+\log I_0(\be|\gt|). $$
This shows that $\varphi_{\beta}(\gt)$ is constant on all the circles centered in 0. 

\begin{remark}
\label{rem-infinitequilquadra}
Unless $\varphi_{\be}$ is maximal only at $0$, which does not hold for every $\be$ as we will see below, we have here an example where all the maxima of the auxiliary function are degenerated.
$\blacksquare$\end{remark}

We set 
$$\phi_{\be}(x):=-\frac{\be}{2} x^{2}+\log I_0(\be x)\textrm{ for }x\geq 0.$$ 
The equality \eqref{equ4-cvmunbe} becomes
\begin{multline}
\label{equ4-XY}
\int_\Omega f(\om)\,d\munbe(\om)= \\
\frac{\ds\int_{B(0,K)} e^{n\varphi_{\be}(\gz)}
\left[\int_{\Omega} f\,d\mu_{\beta\gz\cdot\ogpsi}+
\alpha(n,\gz,f)\right]\,d\gz
+O\pare{\frac{e^{-n\beta\frac{K^2}{4}}}{n^{\frac{3}{2}}}}}
{\ds\int_{B(0,K)}e^{n\varphi_{\be}(\gz)}\,d\gz
+O\pare{\frac{e^{-n\beta\frac{K^2}{4}}}{n^{\frac{3}{2}}}}}.
\end{multline}
where we replaced for commodity the square $[-K,K]^2$ by the disk $B(0,K)$.
We are thus led to study the asymptotic behaviour of the integral 
$$I(n,f)=\ds\int_{B(0,K)} e^{n\varphi_{\be}(\gz)}
\left[\int_{\Omega} f\,d\mu_{\beta\gz\cdot\ogpsi}\right]\,d\gz.$$
In polar coordinates we write $\gz=r\vec u_\theta$ and we get
$$I(n,f)=\ds\int_0^K\int_{-\pi}^{\pi} e^{n\phi_{\be}(r)}
\left[\int_{\Omega} f\,d\mu_{\beta r\vec u_{\theta}\cdot\ogpsi}\right]\,r\,dr\,d\theta.$$
For $x \in\bb R_+$, we denote by $\eta_{x}$ the mean value of DGM's defined by
$$\int_{\Omega} h\,d\eta_x=
\frac{1}{2\pi}\int_{-\pi}^{\pi} 
\left[\int_{\Omega} h\,d\mu_{x\vec u_{\theta}\cdot\ogpsi}\right]\,d\theta$$
for any bounded measurable $h$,
so that 
$$I(n,f)=2\pi \ds\int_0^K e^{n\phi_{\be}(r)}\pare{\int_{\Omega} f\,d\eta_{\be r}}\,r\,dr,$$
which is then a one dimensional Laplace integral.
We study the maximum of the function $\phi_\be$ on $\bb R_+$.
First we notice that $0\leq I_0(\be x)\leq \be x$ and $I_0(0)=1$, 
hence 
$$\phi_\be(x)\leq \be x(1-\frac{x}{2})\textrm{ and } \phi_\be(0)=0,$$
from which we deduce that $\underset{\bb R_+}{\max}\,\phi_\be=\underset{[0,2[}{\max}\,\phi_\be$.
Next we look for the critical points of $\phi_\be$ on $[0,2[$. We compute the first and second derivatives
\begin{equation}\label{der-prem}\phi_\be'(x)=\be\left[\pare{\frac{I_0'}{I_0}}(\be x)-x\right] 
=\be\left[ \frac{\int_0^\pi e^{\be x\cos\theta}\cos\theta\,d\theta}{\int_0^\pi e^{\be x\cos\theta}\,d\theta}-x\right],
\end{equation}
\begin{multline}\label{der-sec}\phi_\be''(x)
=\be\left[\be\pare{\frac{I_0''\,I_0-I_0'^2}{I_0^2}}(\be x)-1\right]\\
=\be\left[\frac{\int_0^\pi e^{\be x\cos\theta}\cos^2\theta\,d\theta}{\int_0^\pi e^{\be x\cos\theta}\,d\theta}-
\pare{ \frac{\int_0^\pi e^{\be x\cos\theta}\cos\theta\,d\theta}{\int_0^\pi e^{\be x\cos\theta}\,d\theta}}^2-1\right].
\end{multline}
We notice that $\phi_\be'(x)\leq \be (1-x)$, from which we deduce that
$\underset{\bb R_+}{\max}\,\phi_\be=\underset{[0,1]}{\max}\,\phi_\be$.
As $I_0'(0)=0$ we know that $\phi_\be'(0)=0$.
We compute
$$\phi_\be''(0)=
\be\left[\be I_0''(0)-1\right]
=\be\left[\frac{\be}{\pi} \int_0^\pi \cos^2\theta\,d\theta-1\right]
=\be\left[\frac{\be}{2}-1\right].$$
We shall thus consider three cases: $\be>2$, $\be=2$, and $\be<2$.
First we take a closer look at the critical points of $\phi_\be$.
We recall that the Bessel function $I_0$ satisfies the differential equation
(we refer for instance to \cite{Bowman} for information about Bessel functions)
\begin{equation}\label{edo}
I_0''(x)+\frac{1}{x}I_0'(x)-I_0(x)=0
\end{equation} 
so that 
$$\pare{\frac{I_0''}{I_0}}(\be x)=1-\frac{1}{\be x}\pare{\frac{I_0'}{I_0}}(\be x).$$
Replacing in \eqref{der-sec} we get that for every $x$,
\begin{multline}\label{der-sec2}
\phi_\be''(x)=\be^2\left[\pare{\frac{I_0''}{I_0}}(\be x)-\pare{\frac{I_0'}{I_0}}^2(\be x)-\frac{1}{\be}\right]\\
=-\be^2\left[\pare{\frac{I_0'}{I_0}}^2(\be x)+\frac{1}{\be x}\pare{\frac{I_0'}{I_0}}(\be x)-1+\frac{1}{\be}\right].
\end{multline}
Now from \eqref{der-prem} we know that $r$ is a critical point of $\phi_\be$ if and only if 
$$\pare{\frac{I_0'}{I_0}}(\be r)=r.$$
If $r$ is such a point then replacing in \eqref{der-sec2} we get
\begin{equation}\label{pc}
\phi_\be''(r)=-\be^2\left[r^2+\frac{2}{\be}-1\right].
\end{equation}

\textit{Case $\be>2$:}
In this case $\phi_\be''(0)>0$ hence $0$ is not a maximum point.
We claim that $\phi$ has a unique maximum and that it belongs to $\,]\sqrt{\frac{\be-2}{\be}},1]$.

We denote by $r_1<\cdots<r_m$ the $m$ points of $]0,1]$ where
$\phi_\be$ attains its maximum $M$ on $\bb R_+$.
Then every $r_k$ satisfies $\phi_\be'(r_k)=0$ and $\phi_\be''(r_k)\leq 0$.
Remember that every critical point $r$ satisfies \eqref{pc} which we rewrite
\begin{equation}\label{pc2}\phi_\be''(r)=\be^2\left[\frac{\be-2}{\be}-r^2\right]
=\be^2\pare{\sqrt{\frac{\be-2}{\be}}-r}\pare{\sqrt{\frac{\be-2}{\be}}+r}.
\end{equation}
As $\phi_\be''(r_k)\leq 0$ we deduce that $r_1\geq  \sqrt{\frac{\be-2}{\be}}$.
We observe that any critical point $r$ strictly bigger than $r_{1}$ satisfies $\phi_\be''(r)<0$,
which means $r$ is a local maximum for $\phi_\be$. 
Now, if $m\geq 2$ then between $r_1$ and $r_2$ there must be a local minimum
which is also a critical point. This yields a contradiction.
Therefore, $m=1$ and $r_{1}$ is the unique critical point for $\phi$. 

Let us  show that $r_1>  \sqrt{\frac{\be-2}{\be}}$. Indeed if $r_1=\sqrt{\frac{\be-2}{\be}}$ then 
$\phi_\be'(r_1)=\phi_\be''(r_1)=0$.
From \eqref{der-prem} we know that
$$\forall\,x,\, \pare{\frac{I_0'}{I_0}}(\be x)=x+\frac{\phi_\be'(x)}{\be}.$$
Replacing in \eqref{der-sec2} we get that
\begin{equation}\label{der-sec3}
\forall\,x,\, \phi_\be''(x)
=-\phi_\be'^2(x)-\pare{2\be x+\frac{1}{x}}\phi_\be' (x)-\be^2x^2+\be^2-2\be 
\end{equation}
which is a differential equation satisfied by $\phi_\be$.
When we derive this equality we get that
\begin{equation}\label{der-sec4}
\forall\,x,\, \phi_\be'''(x)
=-2\phi_\be'(x)\phi_\be''(x)
-\pare{2\be-\frac{1}{x^2}}\phi_\be' (x)
-\pare{2\be x+\frac{1}{x}}\phi_\be'' (x)
-2\be^2x.
\end{equation}
We deduce that
$\phi_\be'''(r_1)
=-2\be^2 r_1$ is strictly negative, therefore $r_1$ can not be a maximum, which yields a contradiction. 
Hence $r_1 >\sqrt{\frac{\be-2}{\be}}$ holds and this finishes to prove the claim. 

Now we are ready to conclude the case $\be >2$.
We apply the Laplace method to the integral 
$I(n,f)$ and we find that
$$I(n,f)\sim 2\pi\frac{\sqrt{2\pi} e^{n\phi_{\be}(r)}
r\,(\int_{\Omega} f\,d\eta_{\beta r})}{n^{1/2}\sqrt{|\phi_\be''(r)|}},$$
hence
$$\int_\Omega f(\om)\,d\munbe(\om)\sim 
\frac{I(n,f)}{I(n,\indic)}
\sim \int_{\Omega} f\,d\eta_{\beta r}.
$$

\textit{Case $\be<2$:}
In this case $\phi'_\be(0)=0$ and $\phi_\be''(0)<0$ hence $0$ is a local maximum of $\phi_\be$.
The equation \eqref{pc} tells us that every critical point $r$ satisfies
$$\phi_\be''(r)=-\be^2\left[r^2+\frac{2-\be}{\be}\right],$$
which
is strictly negative, therefore every critical point is a local maximum. The same argument than above shows that $\phi_\be$ attains its maximum at $0$ and only at $0$.

As the maximum is reached at 0, we cannot directly apply Laplace method as it is emphasized in Remark \ref{rem-gnonnul}. 
We then use the following lemma. It is a special version of Laplace method and can be found in \cite{Dieudonne}.
\begin{lemma}\label{asymp}
Let $\al$ and $\ga$ be two positive real numbers. Then, for any  sequence $(b_{n})$ such that $nb_{n}^{\al}\to+\8$
$$\int_{0}^{b_{n}}x^{\ga}e^{-nx^{\al}}dx\sim_{\ninf}\frac1{\al\,n^{\frac{\ga+1}\al}}\Gamma\left(\frac{\ga+1}\al\right).$$
\end{lemma}
\begin{proof}
Just set $u=nx^{\al}$. 
\end{proof}

We apply Lemma \eqref{asymp}  for $\disp\int_0^{b_{n}}e^{n\phi_{\be}(r)}\pare{\int_{\Omega} f\,d\eta_{\be r}}\,r\,dr,$, with $b_{n}=1/\sqrt[4] n$.  Because $b_{n}\to0$, we can get  $\phi_{\be}(r)=-\phi_{\be}''(0)r^{2}+O(r^{3})$ on $[0,b_{n}]$.   
Note that for $r\in [0,b_{n}]$, by continuity for the  eigen-measures  $\nu_{\be\gt\cdot\ogpsi}$ for operators $\CL_{\gt\cdot\ogpsi}$ we also have. 
$$r\int f\,d\eta_{\be r}\sim_{\ninf} r\int f\,d\eta_{0}$$
A computation shows that $\disp \int_{b_{n}}^{b}e^{n\phi_{\be}(r)}\pare{\int_{\Omega} f\,d\eta_{\be r}}\,r\,dr,$
is of order less than $e^{-nb_{n}^{2}/2}$ if $b$ is chosen small but positive such that $\phi_{\be}(r)\le -\phi''_{\be}(0)\frac{r^{2}}2$ on $[0,b]$. Then, $nb_{n}^{2}\to+\8$ yields that this quantity is exponentially small (in $n$). Because $0$ is the unique maximum for $\phi_{\beta}$, $\disp \int_{b_{n}}^{b}e^{n\phi_{\be}(r)}\pare{\int_{\Omega} f\,d\eta_{\be r}}\,r\,dr,$ is of order less than $e^{-n\eps(b)}$ with $\eps(b)>0$.

Hence we get 
$$I(n,f)\sim 2\pi\frac{ \int_{\Omega} f\,d\eta_{0}}{n|\phi_\be''(0)|}\Gamma(1)=
2\pi\frac{ \int_{\Omega} f\,d\mu_{0}}{n|\phi_\be''(0)|}
,$$
hence
$$\int_\Omega f(\om)\,d\munbe(\om)\sim 
\frac{I(n,f)}{I(n,\indic)}
\sim \int_{\Omega} f\,d\mu_{0}.
$$

\textit{Case $\be=2$:}
In this case $\phi'_\be(0)=\phi_\be''(0)=0$ and 
every critical point $r\neq 0$ is a local maximum since it satisfies 
$$\phi_\be''(r)=-\be^2r^2,$$
which is strictly negative.
We deduce, as above, that there exists only one maximum, which may be 0 or not. 
If it is 0, then we conclude as before but we have to pick an higher order for the derivative (and use Lemma \ref{asymp} with $\al\ge 4$). If it is not 0, then we conclude the computation as above.


We conclude that for every $\be>0$, the sequence of measures 
$(\munbe)_{n\in\bb N}$ weakly converges to $\eta_{\be r}$ where $r$ is the unique point at which $\phi_\be$ reaches its maximum on $\bb R^+$.
\bibliographystyle{plain}

\end{document}